\documentclass[final,1p,times]{elsarticle}

\usepackage{amsmath, epsfig, cite}
\usepackage{amssymb}
\usepackage{amsfonts}
\usepackage{latexsym}
\usepackage{graphicx}
\usepackage{amsthm,array}
\usepackage{color,xcolor}
\usepackage{algorithm}
\usepackage{algorithmic,float}
\usepackage{setspace}

\usepackage{array}

\newtheorem{theorem}{Theorem}[section]
\newtheorem{proposition}[theorem]{Proposition}
\newtheorem{lemma}[theorem]{Lemma}
\newtheorem{example}[theorem]{Example}

\DeclareMathOperator{\Spr}{Spr}
\DeclareMathOperator{\Dis}{Dis}
\makeatletter
\newenvironment{breakablealgorithm}
  {
   \begin{center}
     \refstepcounter{algorithm}
     \hrule height.8pt depth0pt \kern2pt
     \renewcommand{\caption}[2][\relax]{
       {\raggedright\textbf{\ALG@name~\thealgorithm} ##2\par}%
       \ifx\relax##1\relax 
         \addcontentsline{loa}{algorithm}{\protect\numberline{\thealgorithm}##2}%
       \else 
         \addcontentsline{loa}{algorithm}{\protect\numberline{\thealgorithm}##1}%
       \fi
       \kern2pt\hrule\kern2pt
     }
  }{
     \kern2pt\hrule\relax
   \end{center}
  }
\makeatother

\journal{ }

\begin{document}

\begin{frontmatter}



\title{Rational Solutions to the First Order Difference Equations in the Bivariate Difference Field}
\author[1]{Qing-Hu Hou}\ead{qh\_hou@tju.edu.cn}

\author[2]{Yarong Wei\corref{cor1}}\ead{yarongwei@email.tjut.edu.cn}
\address[1]{School of Mathematics, Tianjin University, Tianjin, 300350, China}
\address[2]{School of Science, Tianjin University of Technology, Tianjin 300384, PR China}
\cortext[cor1]{Corresponding author}

\begin{abstract}
Inspired by Karr's algorithm, we consider the summations involving a sequence satisfying a recurrence of order two. The structure of such summations provides an algebraic framework for solving
the difference equations of form $a\sigma(g)+bg=f$ in the bivariate difference field $(\mathbb{F}(\alpha, \beta), \sigma)$, where $a, b,f\in\mathbb{F}(\alpha,\beta)\setminus\{0\}$ are known binary functions of $\alpha$, $\beta$, and $\alpha$, $\beta$ are two algebraically independent transcendental elements, $\sigma$ is a transformation that satisfies
$\sigma(\alpha)=\beta$, $\sigma(\beta)=u\alpha+v\beta$, where $u,v\neq 0\in\mathbb{F}$. Based on it, we then describe algorithms for finding the universal denominator for those equations in the bivariate difference field under certain assumptions. This reduces the general problem of finding the rational solutions of such equations to the problem of finding the polynomial solutions of such equations.
\end{abstract}




\end{frontmatter}


\section{Introduction}
In dealing with summations involving harmonic numbers, Karr \citep{Karr-1981} considered the $\Pi\Sigma$-extension difference field $(\mathbb{F}(t), \sigma)$ of the difference field $(\mathbb{F},\sigma)$ with
\begin{align}\label{eq-1-0}
\sigma(t)=a\cdot t+b
\end{align}
for some $a\in\mathbb{F}\setminus\{0\},b\in\mathbb{F}$. A sequence satisfying a recurrence of order one can be easily described by this model.
This consideration could not only deal with hypergeometric terms, like Gosper's algorithm \citep{Gosper-1978}, and $q$-hypergeometric terms, like \citep{Paule-1997}, but also with terms including the harmonic numbers.

There were several extensions of Karr's algorithm. Bronstein \citep{Bronstein-2000} described how to compute the solutions of generalized difference equations  from the point of view of differential fields. Abramov, Bronstein, Petkov\v{s}ek and Schneider \citep{Abramov-2021} gave a complete algorithm to compute all hypergeometric solutions of homogeneous linear difference equations and rational solutions of parameterized linear difference equations in the $\Pi\Sigma^*$-fields.

We focus on sequences satisfying recurrence of higher order. Take, for example, the Fibonacci sequence $F_n$ which satisfies
\[ F_{n+2} = F_{n+1} + F_n, \quad F_0=F_1=1. \]
We regard the summations involving $F_n$ lie in the difference field $(\mathbb{R}(F_{n}, F_{n+1}), \sigma)$ with
\[
\sigma|_{\mathbb{R}} = {\rm id}, \quad  \sigma(F_n)=F_{n+1},  \quad\text{and}\quad \sigma(F_{n+1})=F_n+F_{n+1}.
\]
With this consideration, one can easily find the right hand side of the following identity,
\[\sum_{n=0}^mF_{n+1}^2=F_{m+1}F_{m+2},\]
which is a result of Lucas, published by Koshy \citep[Eq(5.5)]{Koshy-2001}.

Another example is the Lucas sequences. The summations involving it could be considered in the difference field $(\mathbb{R}(\alpha, \beta), \sigma)$ with integer parameters $P, Q$,
\[
\sigma|_{\mathbb{R}} = {\rm id}, \quad  \sigma(\alpha)=\beta,  \quad\text{and}\quad \sigma(\beta)=-Q\alpha+P\beta.
\]
There are two kinds of Lucas sequences: the first kind $U_n(P,Q)$ and the second kind $V_n(P,Q)$  with the same recurrence relation \[a_{n+2}=P\cdot a_{n+1}-Q\cdot a_{n}\]
but difference initial values
\[
U_{0}(P,Q)=0,\ \  U_{1}(P,Q)=1,
\]
and
\[
V_{0}(P,Q)=2,\ \ V_{1}(P,Q)=P.
\]
Note that $U_n(1,-1)$ is the $n$-th Fibonacci number $F_n$, $V_n(1,-1)$ is the $n$-th Lucas number, $U_n(2,-1)$ is the $n$-th Pell number and $V_n(2,-1)$ is the $n$-th Pell-Lucas number.

This consideration leads to a new approach to the summations involving a sequence satisfying a recurrence of order two.
Let $(\mathbb{F},\sigma)$ be a difference field, i.e., $\sigma$ is an automorphism of $\mathbb{F}$. We define the \emph{bivariate difference field} extension $(\mathbb{F}(\alpha, \beta), \sigma)$ of $(\mathbb{F},\sigma)$ to be the field with
\begin{enumerate}
\item $\alpha,\beta$ being algebraically independent transcendental elements over $\mathbb{F}$.
\item $\sigma$ being an automorphism of $\mathbb{F}(\alpha,\beta)$ and
\begin{align}\label{eq-sigma}
\sigma(\alpha)=\beta, \ \ \ \sigma(\beta)=u\alpha+v\beta,
\end{align}
where $v\in\mathbb{F}, u\in\mathbb{F}\setminus\{0\}$.
 \end{enumerate}
Motivated by the Karr's algorithm, we focus on the difference equation of order one
\begin{equation} \label{eq-dif}
 a(\alpha, \beta)\sigma(g(\alpha, \beta))+b(\alpha, \beta)g(\alpha, \beta)=f(\alpha, \beta),
\end{equation}
where $a(\alpha, \beta), b(\alpha, \beta), f(\alpha, \beta)$ are given non-zero rational functions of $\alpha,\beta$.

For the bivariate difference field extension $(\mathbb{F}(\alpha,\beta),\sigma)$,
we will show there are two variables $\overline{\alpha},\overline{\beta}\in\mathbb{F}(\alpha,\beta)$ such that $\mathbb{F}(\alpha,\beta)=\mathbb{F}(\overline{\alpha})(\overline{\beta})$ under certain conditions, where $\mathbb{F}(\overline{\alpha})(\overline{\beta})$ is obtained by adding $\overline {\alpha}$ and $\overline{\beta} $ step by step through $\Pi$-extension of the field $\mathbb{F}$. But in this case, the field $\mathbb{F}$ needs to be algebraically closed, for instant, we need to introduce $\sqrt{5}$ for summation involving Fibonacci number $F_n$.
However we can obtain the properties of $\mathbb{F}(\alpha,\beta)$ according to its properties in the algebraic closure $\overline{\mathbb{F}}(\alpha,\beta)$, which could get directly and avoid introducing the algebraic closed field.

Base on this consideration and the approach, we will give an algorithm to find rational solutions of the difference equation (\ref{eq-dif}) in the difference field $(\mathbb{F}(\alpha,\beta),\sigma)$ under certain assumptions. Recall that a polynomial $p(\alpha,\beta)$ is called a \emph{universal denominator} of (\ref{eq-dif}) if for every rational solution $y(\alpha,\beta)$ to (\ref{eq-dif}), $p(\alpha,\beta)y(\alpha,\beta)$ is a polynomial.
Actually, we only need to find a universal denominator of (\ref{eq-dif}), then it is easy to find a rational solution of (\ref{eq-dif}) by following the Guan and Hou's algorithm \citep{Guan-2021} to find the polynomial solution.

The outline of the article is as follows. In Section \ref{sec-2}, we introduce the basic properties of the bivariate difference filed and some results which would be used to compute the universal denominator of (\ref{eq-dif}), also we show the convenience of this consideration.
Then we give an algorithm to compute a rational solution for the trivial case $a,b\in\mathbb{F}$ under certain assumptions in Section \ref{sec-3}, like Gosper's algorithm. In particular, we explain strictly that why the dispersion between polynomials in the Gosper's representation is finite with respect to the difference field $(\mathbb{F}(\alpha,\beta),\sigma)$, and show that the given upper bound on the degree for the universal denominator of the Gosper's function is sufficient to solve many functions.
Finally, in Section \ref{sec-4}, we present an algorithm on how to find a universal denominator for the nontrivial case $a,b\in\mathbb{F}(\alpha,\beta)$ under certain assumptions. Moreover, we get the finite part of the universal denominator and show the infinite part of the universal denominator based on the given upper bound.
\section{Bivariate Difference Field}\label{sec-2}
In this section, we give some notations and basic results about the bivariate difference field $(\mathbb{F}(\alpha, \beta), \sigma)$.

Let $(\mathbb{F}(\alpha, \beta), \sigma)$ be a bivariate difference field extension of the difference field $(\mathbb{F},\sigma)$ defined by (\ref{eq-sigma}), and let $\mathbb{F}(\alpha, \beta)^*=\mathbb{F}(\alpha, \beta)\setminus\{0\}$. Denote $\mathbb{F}[\alpha, \beta]$ the set of polynomials of $\alpha, \beta$. For a polynomial $f \in \mathbb{F}[\alpha, \beta]$, let $\deg f$ denote the \emph{total degree} in variables $\alpha, \beta$. Moreover, we use the total degree lexicographical ordering so that $\alpha \succ \beta$. We could have the following result immediately.

\begin{theorem}\label{thm-2-0}
Let $(\mathbb{F}(\alpha, \beta), \sigma)$ be a bivariate difference field extension of the difference field $(\mathbb{F},\sigma)$. Then, we have

$(i)$ $\sigma$ preserves the total degree of a polynomial.

$(ii)$ $\sigma$ is an isomorphism from $\mathbb{F}(\alpha,\beta)$ to itself.

$(iii)$ If $t\in\mathbb{F}[\alpha,\beta]\setminus\mathbb{F}$ is an irreducible polynomial, then $\sigma^m t\in\mathbb{F}[\alpha,\beta]\setminus\mathbb{F}$ is also an irreducible polynomial for any $m\in \mathbb{N}$.
\end{theorem}

We will use the following notations that are revelent in Karr's work \citep{Karr-1985} and that have been refined and explored further in Brostein \citep{Bronstein-2000}.

An element $a\in(\mathbb{F}(\alpha, \beta), \sigma)$ is called \emph{invariant} if $\sigma a=a$, \emph{periodic} if $\sigma^na=a$ for some integer $n>0$, \emph{semi-invariant} if $\sigma a=ua$ for $u\in\mathbb{F[\alpha,\beta]}^*$ and \emph{semi-periodic} if $\sigma^na=ua$ for $u\in\mathbb{F[\alpha,\beta]}^*$ and some integer $n>0$.
The set
\[\mbox{Const$_{\sigma}(\mathbb{F(\alpha,\beta)})=\{a\in\mathbb{F(\alpha,\beta)}$ such that $\sigma a=a$\}}
\]
is called the \emph{constant subfield} of $\mathbb{F}(\alpha,\beta)$ with respect to $\sigma$, and we write $\mathbb{F}(\alpha,\beta)^{\sigma}$ for the semi-invariants of $\mathbb{F}(\alpha,\beta)$, $\mathbb{F}(\alpha,\beta)^{\sigma^*}=\bigcup_{n>0}\mathbb{F}(\alpha,\beta)^{\sigma^n}$
for its semi-periodic elements.

By the definition we have Const$_{\sigma}(\mathbb{F(\alpha,\beta)})\subseteq\mathbb{F}(\alpha,\beta)^{\sigma}\subseteq\mathbb{F}(\alpha,\beta)^{\sigma^*}$.
Note that since $\sigma$ keeps the total degree of a polynomial unchanged, the definitions of \emph{semi-invariant} and \emph{semi-periodic} are equivalent to: $\sigma a=ua$ for $u\in\mathbb{F}\setminus\{0\}$ and $\sigma^na=ua$ for $u\in\mathbb{F}\setminus\{0\}$ and some integer $n>0$, respectively.
Now, we show a sufficient condition for that $p\in\mathbb{F}(\alpha,\beta)^{\sigma^*}$ if and only if $p=\mathbb{F}(\alpha,\beta)^\sigma$, where $p$ is a homogenous polynomial.
\begin{theorem}\label{thm-2-1}
Let $(\mathbb{F}(\alpha, \beta), \sigma)$ be a bivariate difference field extension of the difference field $(\mathbb{F}, \sigma)$ with parameters $u,v$, and let $A$ be the matrix
\begin{align*}
  A =
  \left(
    \begin{array}{cc}
      0 & u \\
      1 & v \\
    \end{array}
  \right).
  \end{align*}
Suppose that $\sigma|_\mathbb{F}={\rm id}$ and $A$ has two different eigenvalues $\lambda_1,\lambda_2 \in {\mathbb{F}}$. If $\lambda_1/\lambda_2$ is not a root of unity, then for two relatively prime homogeneous polynomials $p,q\in\mathbb{F}[\alpha,\beta]$, we have
\[
p  \in \mathbb{F}(\alpha,\beta)^{\sigma^*} \iff p \in  \mathbb{F}(\alpha,\beta)^\sigma
\]
and
\[
p/q  \in \mathbb{F}(\alpha,\beta)^{\sigma^*} \iff p/q \in  \mathbb{F}(\alpha,\beta)^\sigma
\]
\end{theorem}

\begin{proof}
Notice that when $p,q \in \mathbb{F}[\alpha,\beta]$ are relatively prime,
\[
p/q \in \mathbb{F}(\alpha,\beta)^{\sigma^*} \iff p,q \in  \mathbb{F}(\alpha,\beta)^{\sigma^*}.
\]
Therefore, we need only prove
\[
p  \in \mathbb{F}(\alpha,\beta)^{\sigma^*} \implies p \in  \mathbb{F}(\alpha,\beta)^\sigma
\]
for homogeneous polynomials $p$.

Let $n$ be a nonnegative integer and
\[
V = \left\{ \left. \sum_{i=0}^n c_i \alpha^i \beta^{n-i} \,\right|\, c_i \in \mathbb{F} \right\}.
\]
For any positive integer $k$, the map $\sigma^k \colon p \mapsto \sigma^k(p)$ form a linear transformation of the $\mathbb{F}$-vector space  $V$. In terms of linear transformation, the assertion is equivalent to the claim that  all eigenvectors of $\sigma^k$ are eigenvectors of $\sigma$.

Let $X_1,X_2$ be the eigenvectors corresponding to the eigenvalues $\lambda_1,\lambda_2$, respectively. Let
\[
\overline{\alpha} = (\alpha, \beta) X_1, \quad \overline{\beta} = (\alpha, \beta) X_2.
\]
Then we have
\[
\sigma(\overline{\alpha}) = \sigma((\alpha, \beta) X_1) = (\alpha, \beta) AX_1 = \lambda_1  \overline{\alpha}
\quad \mbox{and} \quad  \sigma(\overline{\beta}) = \lambda_2  \overline{\beta}.
\]
Moreover, for any nonnegative integer $n$, we have
\[
\sigma(\overline{\alpha}^i \overline{\beta}^{n-i}) = \lambda_1^i \lambda_2^{n-i} \cdot \overline{\alpha}^i \overline{\beta}^{n-i}, \quad \sigma^k(\overline{\alpha}^i \overline{\beta}^{n-i}) =  (\lambda_1^i \lambda_2^{n-i})^k \cdot \overline{\alpha}^i \overline{\beta}^{n-i}.
\]
Since $\lambda_1/\lambda_2$ is not a root of unity, the numbers $(\lambda_1^i \lambda_2^{n-i})^k,\, i=0,1,\ldots,n$ are pairwise distinct, and thus $\overline{\alpha}^i \overline{\beta}^{n-i},\, i=0,1,\ldots,n$ are $n+1$ independent eigenvectors of $\sigma^k$. Since $\dim V = n+1$, every eigenvector of $\sigma^k$ is a multiple of one of them, completing the proof.
\end{proof}

Note that for a sequence satisfying a recurrence of order two, the matrix $A$ is determined by the recurrence, and it is easy to check that the eigenvalues $\lambda_1, \lambda_2$ are also the roots of the characteristic polynomial of the sequence. As mentioned in \citep{Everest-2003}, we call a sequence {\it non-degenerate} if $\lambda_1/\lambda_2$ is not a root of unity. The following example shows the condition that $\lambda_1/\lambda_2$ is not a root of unity is necessary.

\begin{example}\label{example-2-0}
Let $\mathbb{C}$ be the field of complex numbers and $(\mathbb{C}(\alpha, \beta), \sigma)$ be a bivariate difference field extension of the difference field $(\mathbb{C}, \sigma)$ with
\[
\sigma|_{\mathbb{C}} = {\rm id},\quad\sigma(\alpha)=\beta,  \quad\text{and}\quad \sigma(\beta)=-\alpha+\beta.
\]
It is easy to compute that
 \[
  A =
  \left(
    \begin{array}{cc}
      0 & -1 \\
      1 & 1 \\
    \end{array}
  \right)
  \]
has two eigenvalues
\[
\lambda_1=\frac{1+\sqrt{3}i}{2}, \quad \lambda_2=\frac{1-\sqrt{3}i}{2},
\]
with $(\lambda_1/\lambda_2)^3=1$. Since
\[
\sigma^3(\beta) = - \beta,
\]
we have $\beta \in \mathbb{C(\alpha,\beta)^{\sigma^*}}$. But
\[
\sigma(\beta) = \beta-\alpha
\]
is not a multiple of $\beta$, which implies
\[\beta\notin\mathbb{C(\alpha,\beta)^{\sigma}}.\]
\qed\end{example}

From the proof of Theorem \ref{thm-2-1}, we see that $\mathbb{F}(\alpha,\beta)=\mathbb{F}(\overline{\alpha},\overline{\beta})$ if matrix $A$ is diagonalizable in the field $\mathbb{F}$. In this case, we also could get $ \mathbb {F} (\alpha, \beta )$ by adding $\overline {\alpha}$ and $\overline{\beta} $ step by step, which is a $\Pi$-extension of the
field
$\mathbb{F}$ each time. But to make $A$ diagonalizable, we often need to introduce quadratic extension.
For instance, we need to introduce $\sqrt{5}$ and $\sqrt{2}$ for summation involving Fibonacci number $F_n$ and summation involving Pell number $P_n$, respectively, see the following examples.
However we can obtain the properties of $\mathbb{F}(\alpha,\beta)$ according to its properties in the algebraic closure $\overline{\mathbb{F}}(\alpha,\beta)$, which could get directly and avoid introducing the algebraic closed field $\mathbb{F}(\sqrt{5})(\sqrt{2})(\alpha,\beta)$ for summation involving Fibonacci number.

\begin{example}\label{example-2-1}
Let $\mathbb{R}$ be the field of real numbers. The summations involving only Fibonacci numbers $F_n$ or only Lucas numbers $L_n$ could be considered in the difference field $(\mathbb{R}(\alpha, \beta), \sigma)$ with
\[
\sigma|_{\mathbb{R}} = {\rm id}, \quad\sigma(\alpha)=\beta,  \quad\text{and}\quad \sigma(\beta)=\alpha+\beta.
\]
It is easy to compute that
\[
  A =
  \left(
    \begin{array}{cc}
      0 & 1 \\
      1 & 1 \\
    \end{array}
  \right)
  \]
has two eigenvalues
\[\lambda_1=\frac{1}{2}+\frac{\sqrt{5}}{2}, \quad \lambda_2=\frac{1}{2}-\frac{\sqrt{5}}{2},\]
with $\lambda_1/\lambda_2$ is not a root of unity.
Then by Theorem \ref{thm-2-1}, we have, for the homogeneous polynomials $p\in\mathbb{F}[\alpha,\beta]$,
\[p\in\mathbb{R(\alpha,\beta)^{\sigma^*}}\iff p\in\mathbb{R(\alpha,\beta)^{\sigma}}.\]
\qed\end{example}

\begin{example}\label{example-2-3}
Let $\mathbb{R}$ be the field of real numbers. The summations involving only Pell numbers $P_n$ or only Pell-Lucas numbers $Q_n$ could be considered in the difference field $(\mathbb{R}(\alpha, \beta), \sigma)$ with
\[
\sigma|_{\mathbb{R}} = {\rm id}, \quad \sigma(\alpha)=\beta,  \quad\text{and}\quad \sigma(\beta)=\alpha+2\beta.
\]
It is easy to compute that
\[
  A =
  \left(
    \begin{array}{cc}
      0 & 1 \\
      1 & 2 \\
    \end{array}
  \right)
  \]
has two eigenvalues
\[\lambda_1=1+\sqrt{2}, \quad \lambda_2=1-\sqrt{2},\]
with $\lambda_1/\lambda_2$ is not a root of unity.
Then by Theorem \ref{thm-2-1}, we have, for the homogeneous polynomials $p\in\mathbb{F}[\alpha,\beta]$,
\[p\in\mathbb{R(\alpha,\beta)^{\sigma^*}} \iff p\in\mathbb{R(\alpha,\beta)^{\sigma}}.\]
\qed\end{example}

The key quantity needed to compute universal denominator of the difference equation is the dispersion, which was first introduced with respect to the shift by Abramov \citep{Abramov-1971} for computing rational sums. Now, we show some results of the dispersion in bivariate difference field $(\mathbb{F}(\alpha, \beta), \sigma)$.

Recall that for any two polynomials $p,q\in\mathbb{F}[\alpha,\beta]\backslash\{0\}$, $\Spr_{\sigma}(p,q)$ denotes the \emph{spread} of $p$ and $q$ with respect to $\sigma$, i.e.,
\[\Spr_{\sigma}(p,q)=\{m\in \mathbb{N}|  \deg\left(\gcd(p,\sigma^{m}(q))\right)>0\},\]
$\Dis_{\sigma}(p,q)$ denotes the \emph{dispersion} of $p$ and $q$ with respect to $\sigma$, i.e.,
\[
\mbox{$\Dis_{\sigma}(p,q)$}=\begin{cases}
-1 & \mbox{if $\Spr_{\sigma}(p,q)$ is empty,}\\
\mbox{$\max(\Spr_{\sigma}(p,q)$)} & \mbox{if $\Spr_{\sigma}(p,q)$ is finite and nonempty,}\\
+\infty & \mbox{if $\Spr_{\sigma}(p,q)$ is infinite.}
\end{cases}
\]
Also $\Spr_{\sigma}(p)$ and $\Dis_{\sigma}(p)$ denote $\Spr_{\sigma}(p,p)$ and $\Dis_{\sigma}(p,p)$, respectively.
For a fraction $f\in\mathbb{F}(\alpha,\beta)\backslash\{0\}$, let
\[
\Dis_{\sigma}(f)= \max(\Dis_{\sigma}(p), \Dis_{\sigma}(p,q), \Dis_{\sigma}(q,p), \Dis_{\sigma}(q)),
\]
 where $p,q\in\mathbb{F}[\alpha,\beta]\backslash\{0\}$ such that $f=p/q$, $q$ is monic and $\gcd(p,q)=1$.

Note that $0\in\Spr_{\sigma}(p)$ for any element $p\in \mathbb{F}(\alpha, \beta)$, which implies that  $\Dis_{\sigma}(p)\geq0$. Furthermore $\Dis_{\sigma}(p)=+\infty$, if $p\in\mathbb{F(\alpha,\beta)^{\sigma^*}}$.

The computations of the dispersion in the difference field $(\mathbb{F}(\alpha, \beta), \sigma)$ can be reduced to squarefree polynomials, and the dispersion of a product is related to the dispersions of its components, which is analogous to Lemma 15 in
\citep{Bronstein-2000}.

\begin{lemma}\label{lem-spr}
Let $(\mathbb{F}(\alpha, \beta), \sigma)$ be a bivariate difference field extension of the difference field $(\mathbb{F},\sigma)$ and $p_1,\cdots,p_n,$ $q_1,\cdots,q_m\in\mathbb{F}[\alpha, \beta]\backslash\{0\}$. Then,

$(i)$ for any integers $e_i,f_j>0$,
\[
\Spr_{\sigma}(p_1^{e_1}\cdots p_n^{e_n},q_1^{f_1}\cdots q_m^{f_m})=\bigcup_{i=1}^n\bigcup_{j=1}^m\Spr_{\sigma}(p_i,q_j)\ \bigg(=\Spr_{\sigma}(p_1\cdots p_n,q_1\cdots q_m)\bigg)
\]
and
\[
\Dis_{\sigma}(p_1^{e_1}\cdots p_n^{e_n},q_1^{f_1}\cdots q_m^{f_m})=\max_{1\leq i\leq n,1\leq j\leq m}(\Dis_{\sigma}(p_i,q_j))\ \bigg(=\Dis_{\sigma}(p_1\cdots p_n,q_1\cdots q_m)\bigg).
\]

$(ii)$ if  $\Spr_{\sigma}(p_i,p_j)$ is empty for $i\neq j$, then for any integers $e_i>0$,
\[
\Spr_{\sigma}(p_1^{e_1}\cdots p_n^{e_n})=\bigcup_{i=1}^n \Spr_{\sigma}(p_i)\qquad \text{and} \qquad \Dis_{\sigma}(p_1^{e_1}\cdots p_n^{e_n})=\max_{1\leq i\leq n}(\Dis_{\sigma}(p_i)).
\]

$(iii)$ if $n=m$, $\Spr_{\sigma}(p_i,p_j)$, $\Spr_{\sigma}(q_i,q_j)$, $\Spr_{\sigma}(p_i,q_j)$ and $\Spr_{\sigma}(q_j,p_i)$ are all empty for $i\neq j$, and $\gcd(p_i, q_i)$=1 for all $i$, then for any integers $e_i>0$,
\vspace{2mm}
\[
\Dis_{\sigma}(f_1^{e_1}\cdots f_n^{e_n})=\max_{1\leq i\leq n}(\Dis_{\sigma}(f_i)),
\]
where $f_i=p_i/q_i$.
\end{lemma}
\begin{proof}
The proof is entirely analogous to that of Lemma 15 in
\citep{Bronstein-2000}, and is omitted.
\end{proof}

We now characterize the polynomials whose dispersion is infinite, which is analogous to Theorem 6 in \citep{Bronstein-2000}.

\begin{theorem}{\label{thm-dis}}
Let $(\mathbb{F}(\alpha, \beta), \sigma)$ be a bivariate difference field extension of the difference field $(\mathbb{F},\sigma)$ and $\mathbb{F}[\alpha, \beta]$ be the polynomial ring over a unique factorization domain $\mathbb{F}$.
\begin{enumerate}
\item[$(i)$] Let $p\in\mathbb{F}[\alpha, \beta]\setminus \mathbb{F}$ and $m\in \mathbb{N^+}$ be such that $p|\sigma^mp$. Then, $\Dis_{\sigma}(h\sigma^np,gp)=+\infty$ for any $g,h\in \mathbb{F}[\alpha,\beta]\setminus\{0\}$ and $n\in \mathbb{Z}$.
\item[$(ii)$] Let $q,r\in\mathbb{F}[\alpha, \beta]\setminus\{0\}$. Then, $\Dis_{\sigma}(q,r)=+\infty$ if and only if $r$ has a nontrivial factor $p\in\mathbb{F}[\alpha,\beta]^{\sigma^*}\setminus\mathbb{F}$ such that $\sigma^np|q$ for some integer $n\geq0$.
\item[$(iii)$] Let $q\in \mathbb{F}[\alpha,\beta]\setminus\{0\}$ and suppose that $\sigma$ is an automorphism mapping $\mathbb{F}$ onto $\mathbb{F}$. Then, $\Dis_{\sigma}(q)=+\infty$ if and only if $q$ has a nontrivial factor $p\in\mathbb{F}[\alpha,\beta]^{\sigma^*}\setminus\mathbb{F}$.
\end{enumerate}
\end{theorem}

\begin{proof}
The proof is entirely analogous to that of Theorem 6 in
\citep{Bronstein-2000}, and is omitted.
\end{proof}

Actually, for a given polynomial with infinite dispersion, the dispersions of its irreducible factors are not all infinite. We recall the definition about the splitting factorization, which was firstly defined by Bronstein \citep{Bronstein-1997} for the differential case and generalized in \citep{Bronstein-2000}. For any $p\in \mathbb{F}[\alpha,\beta]$, we say that $p_{\infty}\overline{p}$ is \emph{a splitting factorization} of $p$ with respect to $\sigma$ if $p_{\infty},\overline{p}\in \mathbb{F}[\alpha,\beta]$, $\Dis_{\sigma}(q)=+\infty$ for every irreducible factor $q$ of $p_{\infty}$, and $\Dis_{\sigma}(q)\in\mathbb{Z}$ for every irreducible factor $q$ of $\overline{p}$. We call $\overline{p}$ and $p_{\infty}$ the finite and infinite parts of $p$, respectively.
According to the definition, we show a sufficient condition for finite part in the splitting factorization.

\begin{theorem}\label{thm-prime}
Let $(\mathbb{F}(\alpha, \beta), \sigma)$ be a bivariate difference field extension of the difference field $(\mathbb{F},\sigma)$ and $a,b\in \mathbb{F}[\alpha,\beta]$ be two polynomials. Then $\Spr_{\sigma}(\overline{a},\overline{b})$ is finite, where $\overline{a},\overline{b}$ are the finite parts of $a,b$, respectively.
\end{theorem}

\begin{proof}
Suppose $\Dis_{\sigma}(\overline{a},\overline{b})=+\infty$, let $\overline{a}=u\prod_{j=1}^M\overline{a}_j^{e_j}$ and $\overline{b}=v\prod_{j=1}^N\overline{b}_j^{f_j}$ be the irreducible factorizations of $\overline{a}$ and $\overline{b}$, respectively, where $u,v\in \mathbb{F}\setminus\{0\}$, $\overline{a}_j,\overline{b}_j\in\mathbb{F}[\alpha, \beta]\backslash\mathbb{F}$, $M,N\in\mathbb{N^+}$, and $e_j,f_j>0$. By Lemma \ref{lem-spr}, $S_{s,t}=\Spr_{\sigma}(\overline{a}_s,\overline{b}_t)$ must be infinite for some pair $s,t$. Since $\overline{a}_s$ and $\overline{b}_t$ are irreducible, for each $m,n\in S_{s,t}$, there exists $u_m,u_n\in \mathbb{F}$ such that $\sigma^{m}\overline{b}_t=u_m\overline{a}_s$ and $\sigma^{n}\overline{b}_t=u_n\overline{a}_s$ by Theorem \ref{thm-2-0}, which implies $\sigma^{-m}(\overline{a}_s)\mid \overline{b}_t$, $\sigma^{-n}(\overline{a}_s)\mid \overline{b}_t$ and
\[
\overline{a}_s\in\mathbb{F}(\alpha,\beta)^{\sigma^*},
\]
a contradiction.
\end{proof}

Then, we give one result about infinite part in the splitting factorization, which also be used in computing the universal denominator of the difference equation (\ref{eq-dif}).


\begin{theorem}\label{thm-2-2}
Let $(\mathbb{F}(\alpha, \beta), \sigma)$ be a bivariate difference field extension of the difference field $(\mathbb{F}, \sigma)$ with $\sigma|_\mathbb{F}={\rm id}$ and parameters $u,v$, and let $A$ be the matrix
\begin{align*}
  A =
  \left(
    \begin{array}{cc}
      0 & u \\
      1 & v \\
    \end{array}
  \right).
  \end{align*}
If $A$ has two different eigenvalues $\lambda_1,\lambda_2\in\mathbb{F}$ and $\lambda_1/\lambda_2$ is not a root of unity, then for the homogeneous polynomial $p\in\mathbb{F}(\alpha,\beta)^{\sigma}$ with $\deg p=m$, there exist $c$ and $0\leq i\leq m$ such that
\begin{align}\label{eq-2-1}
p=c\cdot h_1^{i}\cdot h_2^{m-i}
\end{align}
where $h_1(\alpha,\beta)=(\alpha,\beta)X_1$, $h_2(\alpha,\beta)=(\alpha,\beta)X_2$ and $X_1,X_2$ are the eigenvectors corresponding to the eigenvalues $\lambda_1,\lambda_2$, respectively.
\end{theorem}
\begin{proof}
According to the proof of Theorem \ref{thm-2-1}, it could be gotten directly.
\end{proof}

\begin{example}\label{example-2-2}
For the Example \ref{example-2-1}, the bivariate difference field
$(\mathbb{R}(\alpha, \beta), \sigma)$  with $\sigma|_\mathbb{R}={\rm id}$, and
 \[
  A =
  \left(
    \begin{array}{cc}
      0 & 1 \\
      1 & 1 \\
    \end{array}
  \right).
  \]
We can get the two eigenvectors
\[X_1=\left(1,\frac{1}{2}+\frac{\sqrt{5}}{2}\right)^T,\ X_2=\left(1,\frac{1}{2}-\frac{\sqrt{5}}{2}\right)^T.
\]
Then for the homogeneous polynomial $p\in\mathbb{R}[\alpha,\beta]^{\sigma}$ with $\deg p=m$, there exist a $c\in\mathbb{F}$ and $1\leq i\leq m$ such that
\begin{align*}
p=c\cdot \left(\alpha+\left(\frac{1}{2}+\frac{\sqrt{5}}{2}\right)\beta\right)^{i}\cdot \left(\alpha+\left(\frac{1}{2}-\frac{\sqrt{5}}{2}\right)\beta\right)^{m-i}.
\end{align*}
\qed\end{example}

\begin{example}\label{example-2-4}
For the Example \ref{example-2-3}, the bivariate difference field
$(\mathbb{R}(\alpha, \beta), \sigma)$  with $\sigma|_\mathbb{R}={\rm id}$, and
 \[A =
  \left(
    \begin{array}{cc}
      0 & 1 \\
      1 & 2 \\
    \end{array}
  \right).
  \]
We can get the two eigenvectors
\[X_1=\left(1,1+\sqrt{2}\right)^T,\ X_2=\left(1,1-\sqrt{2}\right)^T.
\]
Then for the homogeneous polynomial $p\in\mathbb{R}[\alpha,\beta]^{\sigma}$ with $\deg p=m$, there exist a $c\in\mathbb{F}$ and $1\leq i\leq m$ such that
\[
p=c\cdot \left(\alpha+(1+\sqrt{2})\beta\right)^{i}\cdot \left(\alpha+(1-\sqrt{2})\beta\right)^{m-i}.
\]
\qed
\end{example}
\section{Trivial Case}\label{sec-3}
In this section, for the bivariate difference field $(\mathbb{F}(\alpha,\beta),\sigma)$ with $\sigma|_\mathbb{F}={\rm id}$ and the ratio $\lambda_1/\lambda_2\in\mathbb{F}$ of eigenvectors of matrix $\left(
    \begin{array}{cc}
      0 & u \\
      1 & v \\
    \end{array}
  \right)$ is not root of unity, we consider the trivial case $a, b\in\mathbb{F}$ in the difference equation (\ref{eq-dif}), i.e., for a rational function $f(\alpha, \beta)\in\mathbb{F}(\alpha, \beta)$ and $a, b\in\mathbb{F}$, finding a rational function $g(\alpha,\beta)\in\mathbb{F}(\alpha, \beta)$ such that
\begin{align}\label{eq-3-01}
a\sigma\left(g(\alpha, \beta)\right)+bg(\alpha, \beta)=f(\alpha, \beta),
\end{align}
We will give an algorithm for finding $g(\alpha,\beta)$ and show some examples.

Note the difference equation (\ref{eq-3-01}) with $a=1,b=-1$ is same as the indefinite summation problem. Motivated by the Gosper's algorithm, we want to reduce the problem of finding a rational solution of (\ref{eq-3-01}) to the problem of finding a polynomial solution of some equation.
Unfortunately, we have not reduced successfully, since there are many polynomials such that their dispersion is $+\infty$.
But we will reduce the problem of finding a rational solution of (\ref{eq-3-01}) to the problem of finding a rational solution of some equation, in which the denominator of any reduced rational solution only has infinite part.

Assume $g(\alpha, \beta)=y(\alpha, \beta)f(\alpha, \beta)$ is a solution of (\ref{eq-3-01}), then
\[a\sigma\left(y(\alpha, \beta)\right)\frac{\sigma\left(f(\alpha, \beta)\right)}{f(\alpha, \beta)}+by(\alpha, \beta)=1.\]
If $\frac{\sigma\left(f(\alpha, \beta)\right)}{f(\alpha, \beta)}$ could be written as the form $\frac{A(\alpha,\beta)}{B(\alpha,\beta)}\cdot\frac{\sigma(C(\alpha,\beta))}{C(\alpha,\beta)}$, where $A(\alpha,\beta)$, $B(\alpha,\beta)$, $C(\alpha,\beta)\in\mathbb{F}[\alpha, \beta]$, then there is a rational function $x(\alpha, \beta)\in\mathbb{F}(\alpha, \beta)$ such that
\[
y(\alpha, \beta)=\frac{\sigma^{-1}\left(B(\alpha, \beta)\right)x(\alpha, \beta)}{C(\alpha, \beta)},\]
and
\[
aA(\alpha,\beta)\sigma\left(x(\alpha,\beta)\right)
+b\sigma^{-1}\left(B(\alpha,\beta)\right)x(\alpha,\beta)=C(\alpha,\beta).
\]

Based on these, we show the steps of the algorithm to compute $g(\alpha,\beta)$:
\begin{breakablealgorithm}
\caption{Compute a Solution of Trivial Case}
\label{alg-1}
\begin{algorithmic}[htb]
\STATE{\textbf{Input:} Parameters $u,v$ of the bivariate difference field $(\mathbb{F}(\alpha,\beta),\sigma)$ and a rational function $f(\alpha, \beta)$.}
\STATE{\textbf{Output:} A rational function $g(\alpha, \beta)$ satisfying (\ref{eq-3-01}), if one exists; FALSE, otherwise.}
\STATE{\textbf{Step 1:} Compute the ratio $r(\alpha,\beta)=\frac{\sigma\left(f(\alpha, \beta)\right)}{f(\alpha, \beta)}\in\mathbb{F}(\alpha, \beta)$.}
\STATE{\textbf{Step 2:} Write $r(\alpha,\beta)=\frac{A(\alpha,\beta)}{B(\alpha,\beta)}\cdot\frac{\sigma(C(\alpha,\beta))}{C(\alpha,\beta)}$ where $A(\alpha,\beta), B(\alpha,\beta), C(\alpha,\beta)\in\mathbb{F}[\alpha, \beta]$ satisfy $\gcd(\overline{A}(\alpha,\beta), \sigma^h\left(\overline{B}(\alpha,\beta))\right)=1$ for all $h\in\mathbb{N}$.}
\raggedright{\STATE{\textbf{Step 3:} Find a nonzero rational solution $x(\alpha, \beta)\in\mathbb{F}(\alpha, \beta)$ such that $aA(\alpha,\beta)\sigma\left(x(\alpha,\beta)\right)+b\sigma^{-1}\left(B(\alpha,\beta)\right)$ $x(\alpha,\beta)=C(\alpha,\beta)$ if one exists; otherwise return ``FALSE'' and stop.}}
\STATE{\textbf{Step 4:} Return ``$\frac{\sigma^{-1}\left(B(\alpha, \beta)\right)x(\alpha, \beta)}{C(\alpha, \beta)} f(\alpha, \beta)$'' and stop.}
  \end{algorithmic}
\end{breakablealgorithm}

Now, we show the details for each steps and some examples in following subsections.
\subsection{Step 2 of Algorithm \ref{alg-1}}
In this subsection, we show how to write $r(\alpha,\beta)$ in the following form
\begin{align}\label{eq-3-02}
 r(\alpha,\beta)=\frac{A(\alpha,\beta)}{B(\alpha,\beta)}\cdot\frac{\sigma(C(\alpha,\beta))}{C(\alpha,\beta)}
\end{align}
where $A(\alpha,\beta), B(\alpha,\beta), C(\alpha,\beta)\in\mathbb{F}[\alpha, \beta]$ are polynomials and the finite part of $A,B$ satisfy
\begin{align}\label{eq-3-03}
\deg \gcd\left(\overline{A}(\alpha,\beta), \sigma^h\left(\overline{B}\left(\alpha,\beta\right)\right)\right)=0,
\end{align}
for all nonnegative integers $h$.

Let $r(\alpha,\beta)=\frac{p(\alpha,\beta)}{q(\alpha,\beta)}$ where $p(\alpha,\beta),q(\alpha,\beta)$ are relatively prime polynomials. Then by Theorem \ref{thm-prime}, we have $\Spr_{\sigma}(\overline{p},\overline{q})$ is finite where $\overline{p}, \overline{q}$ are the finite part of $p, q$, respectively.
We will see that $r(\alpha,\beta)$ can be written as (\ref{eq-3-02}) under the condition (\ref{eq-3-03}). For this, we do the following computation.

\begin{breakablealgorithm}
\caption{Step 2 of Algorithm \ref{alg-1}}
\label{alg-2}
\raggedright{\textbf{Step 2.1:} Let $r(\alpha,\beta)=\frac{p(\alpha,\beta)}{q(\alpha,\beta)}$ where $p(\alpha,\beta),q(\alpha,\beta)$ are relatively prime polynomials.}\\
\raggedright{\textbf{Step 2.2:} Compute $\Spr_{\sigma}(\overline{p},\overline{q})=\{m_1,m_2,\ldots,m_N\}$ and $m_1<m_2<\ldots<m_N$.}\\
\raggedright{\textbf{Step 2.3:} $p_0:=\overline{p}$; $q_0:=\overline{q}$;}
\begin{algorithmic}[htb]
 \FOR{$1\leq j\leq N$}
  \STATE{$s_j(\alpha,\beta):=\gcd(p_{j-1}(\alpha,\beta),\sigma^{m_j}q_{j-1}(\alpha,\beta));$}
  \STATE{$p_{j}(\alpha,\beta):=\frac{p_{j-1}(\alpha,\beta)}{s_j(\alpha,\beta)};
         q_{j}(\alpha,\beta):=\frac{q_{j-1}(\alpha,\beta)}{\sigma^{-m_j}\left(s_j(\alpha,\beta)\right)}$.}
\ENDFOR
\STATE {$A(\alpha,\beta):=p_{\infty}p_N(\alpha,\beta)$;}
\STATE {$B(\alpha,\beta):=q_{\infty}q_N(\alpha,\beta)$;} \STATE{$C(\alpha,\beta):=\prod_{i=1}^{N}\prod_{j=1}^{m_i}\sigma^{-j}s_i(\alpha,\beta)$.}
\end{algorithmic}
\end{breakablealgorithm}

According to the computation, we have
\begin{align*}
\frac{A(\alpha,\beta)}{B(\alpha,\beta)}\frac{\sigma C(\alpha,\beta)}{C(\alpha,\beta)}
&=\frac{p_{\infty}p_N(\alpha,\beta)}{q_{\infty}q_N(\alpha,\beta)}
\cdot
\frac{\sigma\left(\prod_{i=1}^{N}\prod_{j=1}^{m_i}\sigma^{-j}s_i(\alpha,\beta)\right)}
{\prod_{i=1}^{N}\prod_{j=1}^{m_i}\sigma^{-j}s_i(\alpha,\beta)}\\
&=\frac{p_{\infty}p_N(\alpha,\beta)}{q_{\infty}q_N(\alpha,\beta)}
\cdot
\frac{\prod_{i=1}^{N}s_i(\alpha,\beta)}
{\prod_{i=1}^{N}\sigma^{-m_i}s_i(\alpha,\beta)}
\cdot
\frac{\prod_{i=1}^{N}\prod_{j=2}^{m_i}\sigma^{-j+1}s_i(\alpha,\beta)}
{\prod_{i=1}^{N}\prod_{j=1}^{m_i-1}\sigma^{-j}s_i(\alpha,\beta)}\\
&=\frac{p_{\infty}p_0(\alpha,\beta)}{q_{\infty}q_0(\alpha,\beta)}=r(\alpha,\beta).
\end{align*}
so $r(\alpha,\beta)$ can be written as (\ref{eq-3-02}). Moreover, for the splitting factorization of $A,B,C$ with respect to $\sigma$, we have
\[A_{\infty}=p_{\infty},\quad \overline{A}=p_N(\alpha,\beta),\quad B_{\infty}=q_{\infty},\quad \overline{B}=q_N(\alpha,\beta),\quad C=\overline{C}.\]

Now, we show that $A(\alpha,\beta)$, $B(\alpha,\beta)$ and $C(\alpha,\beta)$ produced by this computation satisfy the condition (\ref{eq-3-03}). Set additionally $m_{N+1}:=+\infty$.

\begin{proposition}\label{pro-3-1}
Let $i,j,k,h,N\in\mathbb{N}$, $0\leq k\leq i, j \leq N$ and $h<m_{k+1}$, where $m_{k+1}$ is defined in the Step 2.2. Then
\begin{align}\label{eq-3-05}
\deg\left(\gcd\left(p_i, \sigma^hq_j\right)\right)=0,
\end{align}
where $p_i,q_i$ are defined in the Step 2.3.
\end{proposition}
\begin{proof}
Since $p_i\mid p$ and $q_j\mid q$, it follows that $\gcd(p_i, \sigma^hq_j)$ divides $\gcd(p, \sigma^hq)$, for any $h$. If $h\notin\Spr_{\sigma}(p,q)$ then $\deg(\gcd(p,\sigma^hq))=1$. This proves the assertion when $h\notin\Spr_{\sigma}(p,q)$.

When $h\in\Spr_{\sigma}(p,q)$, we will prove the assertion by induction on $k$. Recall that $\Spr_{\sigma}(p,q)=\{m_1<m_2<\ldots<m_N\}$.

$k=0$: In this case there is nothing to prove, since there is no $h\in\Spr_{\sigma}(p,q)$ such that $h<m_1$.

$k>0$: Assume that the assertion holds for all $h<m_k$. It remains to show that it holds for $h=m_k$. Since $p_i\mid p_k$ and $q_j\mid q_k$, it follows that $\gcd(p_i, \sigma^{m_k}q_j)$ divides $\gcd(p_k, \sigma^{m_k}q_k)$.
Actually,
\[
\gcd\left(p_k,\sigma^{m_k}q_k\right)=\gcd\left(\frac{p_{k-1}}{s_k},\frac{\sigma^{m_k}q_{k-1}}{s_k}\right)=1.
\]
Hence $\gcd(p_i, \sigma^{m_k}q_j)=1$, and this concludes the proof.
\end{proof}

Obviously, setting $i=j=k=N$ in (\ref{eq-3-05}) we see that $\gcd(\overline{A}(\alpha,\beta), \sigma^h \overline{B}(\alpha,\beta))=1$ for all $h\in\mathbb{N}$. Moreover, we have the following result.

\begin{theorem}\label{thm-3-1}
Let $r(\alpha,\beta)\in\mathbb{F}(\alpha,\beta)$ be a nonzero rational function. Then there exist polynomials $A(\alpha,\beta), B(\alpha,\beta), C(\alpha,\beta)\in\mathbb{F}[\alpha,\beta]$ such that
\begin{align}\label{eq-3-06}
r(\alpha,\beta)=\frac{A(\alpha,\beta)}{B(\alpha,\beta)}\cdot
\frac{\sigma \left(C(\alpha,\beta)\right)}{C(\alpha,\beta)},
\end{align}
where
\begin{enumerate}
\item[$(i)$] $\deg\left(\gcd\left(\overline{A}(\alpha,\beta), \sigma^h\left(\overline{B}(\alpha,\beta)\right)\right)\right)=0$ for every nonnegative integer $h$,

\item[$(ii)$] $\deg\left(\gcd(A(\alpha,\beta), C(\alpha,\beta))\right)=0$,

\item[$(iii)$] $\deg\left(\gcd(B(\alpha,\beta), \sigma \left(C(\alpha,\beta)\right)\right)=0$.
\end{enumerate}
\end{theorem}
\begin{proof}
The polynomials $A(\alpha,\beta), B(\alpha,\beta), C(\alpha,\beta)$ are constructed by Step 2 of the Algorithm. We have already shown that (\ref{eq-3-06}) and $(i)$ are satisfied.

$(ii)$ By the definition of the splitting factorization,  $A_{\infty}(\alpha,\beta)$ and $C(\alpha,\beta)$ are relatively prime. If $\overline{A}(\alpha,\beta)$ and $C(\alpha,\beta)$ have a non-constant common factor then so do $p_N(\alpha,\beta)$ and
$\sigma^{-j}s_i(\alpha,\beta)$, for some $i$ and $j$ such that $1\leq i \leq N$ and $1 \leq j \leq m_i$.
Since by definition $\sigma^{m_i-j}q_{i-1}(\alpha,\beta)=\sigma^{m_i-j}q_i(\alpha,\beta)\sigma^{-j}s_i(\alpha,\beta)$, it follows that $p_N(\alpha,\beta)$ and
$\sigma^{m_i-j}q_{i-1}(\alpha,\beta)$ have such a common factor, too. Since $m_i-j<m_i$, this contradicts with Proposition \ref{pro-3-1}. Hence $A(\alpha,\beta)$ and $C(\alpha,\beta)$ are relatively prime.

$(iii)$ Obviously, the finite part of $\sigma \left(C(\alpha,\beta)\right)$ is itself, then by the definition of the splitting factorization, $B_{\infty}(\alpha,\beta)$ and $\sigma \left(C(\alpha,\beta)\right)$ are relatively prime. If $\overline{B}(\alpha,\beta)$ and $\sigma \left(C(\alpha,\beta)\right)$ have a non-constant common factor then so do $q_N(\alpha,\beta)$ and $\sigma^{-j}s_i(\alpha,\beta)$, for some $i$ and $j$ such that $1\leq i \leq N$ and $0\leq j\leq m_i-1$.
Since by definition $\sigma^{-j}p_{i-1}(\alpha,\beta)=\sigma^{-j} p_i(\alpha,\beta)\cdot\sigma^{-j}s_i(\alpha,\beta)$, it follows that $\sigma^{j}q_N(\alpha,\beta)$ and $p_{i-1}(\alpha,\beta)$ have
such a common factor, too. Since $j<m_i$, this contradicts Proposition \ref{pro-3-1}. Hence $B(\alpha,\beta)$ and $\sigma \left(C(\alpha,\beta)\right)$ are relatively prime.
\end{proof}
\begin{example}\label{example-3-0}
For the difference field $\mathbb{R}(\alpha,\beta)$ involving the Fibonacci numbers $F_n$ which is defined in Example \ref{example-2-1}, consider the difference equation (\ref{eq-3-01}) with $a=1,b=-1$.
Given $f(\alpha, \beta)=\frac{\alpha}{\beta(\alpha+\beta)}$, find a solution $g(\alpha, \beta)\in\mathbb{F}(\alpha,\beta)$ such that
\begin{align}\label{eq-3-09}
\sigma\left(g(\alpha, \beta)\right)-g(\alpha, \beta)
=\frac{\alpha}{\beta(\alpha+\beta)}.
\end{align}

According to Step 1, compute the ratio
\[r(\alpha,\beta)=\frac{\sigma\left(f(\alpha, \beta)\right)}{f(\alpha, \beta)}=\frac{\beta^2}{\alpha(\alpha+2\beta)}:=\frac{p(\alpha, \beta)}{q(\alpha, \beta)}.\]

According to Step 2, compute that the splitting factorization of $p,q$ with respect to $\sigma$ are $p=\overline{p}$ and $q=\overline{q}$, respectively, and  $\Spr_{\sigma}(\overline{p}(\alpha, \beta),\overline{q}(\alpha, \beta))=\{1\}$, then
\[
s_1=\gcd\left(p,\sigma q\right)=\beta,\  p_1=p/s_1=\beta,\  q_1=q/\sigma^{-1} s_1=\alpha+2\beta.
\]
hence, we have
\[r(\alpha,\beta)=\frac{\beta}{\alpha+2\beta}\cdot\frac{\sigma\alpha}{\alpha}
:=\frac{A(\alpha, \beta)}{B(\alpha, \beta)}\cdot\frac{\sigma\left(C(\alpha, \beta)\right)}{C(\alpha, \beta)},\]
where $\deg\left(\gcd(A(\alpha,\beta), \sigma^h\left(B(\alpha,\beta)\right)\right)=0$ for all nonnegative integers $h$.
\qed\end{example}

\begin{example}\label{example-3-1}
For the difference field $\mathbb{R}(\alpha,\beta)$ involving the Lacus numbers $L_n$ which is defined in Example \ref{example-2-1}, consider the difference equation (\ref{eq-3-01}) with $a=b=1$.
Given $f(\alpha, \beta)=\frac{1}{(\beta-\alpha)\alpha}$, find a solution $g(\alpha, \beta)\in\mathbb{F}(\alpha,\beta)$ such that
\begin{align}\label{eq-3-12}
\sigma\left(g(\alpha, \beta)\right)+g(\alpha, \beta)
=\frac{1}{(\beta-\alpha)\alpha}.
\end{align}

According to Step 1, compute the ratio
\[r(\alpha,\beta)=\frac{\sigma\left(f(\alpha, \beta)\right)}{f(\alpha, \beta)}=\frac{\beta-\alpha}{\beta}:=\frac{p(\alpha, \beta)}{q(\alpha, \beta)}.\]

According to Step 2,  compute that the splitting factorization of $p,q$ with respect to $\sigma$ are $p=\overline{p}$ and $q=\overline{q}$, respectively, and $\Spr_{\sigma}(p(\alpha, \beta),q(\alpha, \beta))=\emptyset$, then we have
\[r(\alpha,\beta)=\frac{\beta-\alpha}{\beta}\cdot\frac{1}{1}:=\frac{A(\alpha, \beta)}{B(\alpha, \beta)}\cdot\frac{\sigma\left(C(\alpha, \beta)\right)}{C(\alpha, \beta)},\]
where $\deg\left(\gcd(A(\alpha,\beta), \sigma^h\left(B(\alpha,\beta)\right)\right)=0$ for all nonnegative integers $h$.
\qed\end{example}

\begin{example}\label{example-3-2}
For the difference field $\mathbb{R}(\alpha,\beta)$ involving the  Pell numbers $P_n$ which is defined in Example \ref{example-2-3}, consider the difference equation (\ref{eq-3-01}) with $a=1,b=-1$.
Given $f(\alpha, \beta)=\frac{\alpha}{(\beta-2\alpha)\beta}$, find a solution $g(\alpha, \beta)\in\mathbb{F}(\alpha,\beta)$ such that
\begin{align}\label{eq-3-13}
\sigma\left(g(\alpha, \beta)\right)-g(\alpha, \beta)
=\frac{\alpha}{(\beta-2\alpha)\beta}.
\end{align}

According to Step 1, compute the ratio
\[r(\alpha,\beta)=\frac{\sigma\left(f(\alpha, \beta)\right)}{f(\alpha, \beta)}=\frac{\beta^2(\beta-2\alpha)}{\alpha^2(\alpha+2\beta)}:=\frac{p(\alpha, \beta)}{q(\alpha, \beta)}.\]

According to Step 2, compute that $\Spr_{\sigma}(p(\alpha, \beta),q(\alpha, \beta))=\{1\}$, then we have
\[
s_1=\gcd\left(p,\sigma q\right)=\beta^2,\  p_1=p/s_1=\beta-2\alpha,\  q_1=q/\sigma^{-1} s_1=\alpha+2\beta.
\]
hence we have
\[r(\alpha,\beta)=\frac{\beta-2\alpha}{\alpha+2\beta}\cdot\frac{\sigma(\alpha^2)}{\alpha^2}:=
\frac{A(\alpha, \beta)}{B(\alpha, \beta)}\cdot\frac{\sigma\left(C(\alpha, \beta)\right)}{C(\alpha, \beta)},\]
where $\deg\left(\gcd(A(\alpha,\beta), \sigma^h\left(B(\alpha,\beta))\right)\right)=0$ for all nonnegative integers $h$.
\qed\end{example}

\subsection{Step 3 of Algorithm \ref{alg-1}}
In this subsection, we explain how to look for a nonzero solution $x(\alpha, \beta)\in\mathbb{F}(\alpha, \beta)$ of
\begin{align}\label{eq-3-07}
aA(\alpha,\beta)\sigma\left(x(\alpha,\beta)\right)+b\sigma^{-1}\left(B(\alpha,\beta)\right)x(\alpha,\beta)=C(\alpha,\beta).
\end{align}
Actually, (\ref{eq-3-07}) is a special case of nontrivial difference function (\ref{eq-dif}).
Suppose $x(\alpha, \beta)=\frac{p(\alpha, \beta)}{q(\alpha, \beta)}$ is a reduced rational solution of (\ref{eq-3-07}), i.e, $p(\alpha, \beta),q(\alpha, \beta)\in\mathbb{F}[\alpha, \beta]$ are relatively prime. First, we show one property of polynomial $q(\alpha, \beta)$.

\begin{theorem}\label{thm-3-2}
Let $A(\alpha,\beta), B(\alpha,\beta), C(\alpha,\beta)\in\mathbb{F}[\alpha,\beta]$ be constructed by the Step 2 of the algorithm. If $x(\alpha, \beta)=\frac{p(\alpha, \beta)}{q(\alpha, \beta)}$ is a reduced rational solution of (\ref{eq-3-07}), then $q(\alpha, \beta)$ only has the infinite part in the splitting factorization of $q$ with respect to $\sigma$.
\end{theorem}
\begin{proof}
Let $q=q_{\infty}\overline{q}$ be the splitting factorization of $q$ with respect to $\sigma$.
Suppose $\overline{q}$ is a non-constant polynomial. Let $N$ be the largest integer such that $\deg\left(\gcd\left(\overline{q},\sigma^N(\overline{q})\right)\right)>0$. Note $N\geq0$. By (\ref{eq-3-07}), we have
\begin{align}\label{eq-3-08}
aA(\alpha,\beta)\cdot\sigma\left(p(\alpha, \beta)\right)\cdot &q(\alpha, \beta)+b\sigma^{-1}\left(B(\alpha,\beta)\right)\cdot p(\alpha, \beta)\cdot\sigma\left(q (\alpha, \beta)\right)\nonumber\\
=&C(\alpha,\beta)\cdot q(\alpha, \beta)\cdot\sigma\left(q (\alpha, \beta)\right).
\end{align}
Let $u(\alpha,\beta)$ be a non-constant irreducible common divisor
of $\overline{q}$ and $\sigma^N(\overline{q})$. Since $\sigma^{-N}\left(u(\alpha,\beta)\right)\mid \overline{q}$ and $\overline{q}\mid q$,
\[\sigma^{-N}\left(u(\alpha,\beta)\right)\mid q(\alpha,\beta),\]
and it follows from (\ref{eq-3-08}) that
\[
\sigma^{-N}\left(u(\alpha,\beta)\right)\mid b\sigma^{-1}\left(\overline{B}(\alpha,\beta)\right)\cdot p(\alpha, \beta)\cdot\sigma\left(q (\alpha, \beta)\right).
\]

Now $\sigma^{-N}\left(u(\alpha,\beta)\right)$ does not divide $p$ since it divides $q$, which is relatively prime to $p$ by assumption. It also does not divide $\sigma\left(q (\alpha, \beta)\right)$, or else $u(\alpha,\beta)$ would be a
non-constant common factor of $\overline{q}$ and $\sigma^{N+1}(\overline{q})$, contrary to our choice of $N$. Therefore $\sigma^{-N}(u(\alpha,\beta))\mid\sigma^{-1}\left(\overline{B}(\alpha,\beta)\right)$, since $b\in\mathbb{F}$, and hence $\sigma(u(\alpha,\beta))\mid\sigma^{N}\left(\overline{B}(\alpha,\beta)\right)$.

Similarly, it follows from (\ref{eq-3-08}) that
\[
\sigma\left(u(\alpha,\beta)\right)\mid a\overline{A}(\alpha,\beta)\cdot\sigma\left(p(\alpha, \beta)\right)\cdot q(\alpha, \beta).
\]
Again, $\sigma\left(u(\alpha,\beta)\right)$ does not divide $\sigma\left(p(\alpha, \beta)\right)$ by assumption. It also does not divide $q(\alpha, \beta)$, or else $u(\alpha,\beta)$ would be a non-constant common factor of $\sigma^{-1}\left(q(\alpha, \beta)\right)$ and $\sigma^{N}\left(q(\alpha, \beta)\right)$, contrary to our choice of $N$. Hence $\sigma\left(u(\alpha,\beta)\right)\mid \overline{A}(\alpha,\beta)$, since $a\in\mathbb{F}$. But then, by the previous paragraph,
$\sigma\left(u(\alpha,\beta)\right)$ is a non-constant common factor of $\overline{A}(\alpha,\beta)$ and $\sigma^{N}\left(\overline{B}(\alpha,\beta)\right)$, contrary to Theorem \ref{thm-3-1}.

Therefore, $\overline{q}$ has to be a constant polynomial.
\end{proof}

Now, we show how to find a nonzero solution $x(\alpha, \beta)$ of (\ref{eq-3-07}).

\begin{breakablealgorithm}
\label{alg-3}
\caption{Step 3 of Algorithm \ref{alg-1}}
\begin{algorithmic}[htb]
\STATE {\raggedright{\textbf{Step 3.1:} Compute the eigenvectors $X_1, X_2$ of matrix $\left(
    \begin{array}{cc}
      0 & u \\
      1 & v \\
    \end{array}
  \right)$, and $(h_1,h_2):=(\alpha,\beta)(X_1,X_2)$.}}
\STATE {\raggedright{\textbf{Step 3.2:} $m:=\max\{\deg A,\deg(\sigma^{-1}B)\}-\deg C$.}}
\IF{$m\leq0$}
  \STATE{ Find a nonzero polynomial solution $p(\alpha,\beta)$ of
\[
aA(\alpha,\beta)\sigma\left(p(\alpha,\beta)\right)
+b\sigma^{-1}\left(B(\alpha,\beta)\right)p(\alpha,\beta)=C(\alpha,\beta).
\]
If one exists, then $x(\alpha,\beta):=p(\alpha,\beta)$; otherwise return ``FALSE'' and stop.}
\ELSE
  \FOR{each $0\leq k\leq m+1$}\vspace{-5mm}
     \STATE {\FOR{$0\leq i\leq k$}
                  \STATE {$q_i(\alpha,\beta):=h_1^{k-i}h_2^{i};$}\\
                   \STATE {find a nonzero polynomial solution $p(\alpha,\beta)$ such that
\[
aA(\alpha,\beta)\cdot\sigma\left(\frac{p(\alpha,\beta)}{q_i(\alpha,\beta)}\right)
+b\sigma^{-1}\left(B(\alpha,\beta)\right)\cdot\frac{p(\alpha,\beta)}{q_i(\alpha,\beta)}=C(\alpha,\beta).
\]}\ENDFOR}
\ENDFOR
\STATE{ If there exists, then $x(\alpha,\beta):=\frac{p(\alpha,\beta)}{q_i(\alpha,\beta)}$;
otherwise return ``FALSE'' and stop.}
\ENDIF
\end{algorithmic}
\end{breakablealgorithm}
Note that return ``FALSE'' does not mean there is no nonzero solution $x(\alpha, \beta)=\frac{p(\alpha, \beta)}{q(\alpha, \beta)}$ of (\ref{eq-3-07}).
By Theorem \ref{thm-2-2} and Theorem \ref{thm-3-2}, for any solution $x(\alpha, \beta)=\frac{p(\alpha, \beta)}{q(\alpha, \beta)}$, if $q$ is a homogenous polynomial, then there exist $m\in\mathbb{N}$ and $0\leq i\leq m$ such that
\[
q(\alpha, \beta)|h_1(\alpha,\beta)^{i}h_2(\alpha,\beta)^{m-i},
\]
where $(h_1(\alpha,\beta),h_2(\alpha,\beta))=(\alpha, \beta)(X_1,X_2)$ and $X_1, X_2$ are two eigenvectors of matrix
$A=\left(
    \begin{array}{cc}
      0 & u \\
      1 & v \\
    \end{array}
  \right)$.
However we could not get the upper bound of $m$, hence we just assume that $m\leq \max\{\deg A,\sigma^{-1}\deg B\}-\deg C$, then use the method of finding a polynomial solution to compute polynomial $p(\alpha, \beta)$.

\begin{example}
Consider the solution $g(\alpha, \beta)$ in Example \ref{example-3-0}.

According to Step 3, we need to find a nonzero rational solution $x(\alpha, \beta)\in\mathbb{F}(\alpha, \beta)$ such that
\begin{align}\label{eq-3-10}
\beta\cdot\sigma\left(x(\alpha,\beta)\right)-\sigma^{-1}\left(\alpha+2\beta)\right)\cdot x(\alpha,\beta)=\alpha,
\end{align}
if one exists.

Recall that
$\left(h_1,h_2\right)=\left(\alpha+\left(\frac{1}{2}+\frac{\sqrt{5}}{2}\right)\beta,\alpha+\left(\frac{1}{2}-\frac{\sqrt{5}}{2}\right)\beta\right)$, which is computed in Example \ref{example-2-2}.

Since $m=\max\{\deg A,\sigma^{-1}\deg B\}-\deg C=1-1=0$, the problem is reduced to find a nonzero
polynomial solution of (\ref{eq-3-10}), then using the method of finding a polynomial solution, we get $x(\alpha,\beta)=-1$, hence \[g(\alpha,\beta)=\frac{\sigma^{-1}\left(B(\alpha,\beta)\right)\cdot x(\alpha,\beta)}{C(\alpha,\beta)}\cdot f(\alpha,\beta)
=-\frac{1}{\beta}\] is a solution of (\ref{eq-3-09}),
which implies
\[\sum_{n=0}^k\frac{F_n}{F_{n+1}F_{n+2}}=1-\frac{1}{F_{k+2}}.\]
Setting $k\rightarrow\infty$, we obtain
\[\sum_{n=0}^\infty\frac{F_n}{F_{n+1}F_{n+2}}=1,\]
it is a result of Brousseau \citep[Eq.(1)]{Brousseau-1969b}.
\qed
\end{example}

\begin{example}
Consider the solution $g(\alpha, \beta)$ in Example \ref{example-3-1}.

According to Step 3, we need to find a nonzero solutions $x(\alpha, \beta)\in\mathbb{F}(\alpha, \beta)$ such that
\[
(\beta-\alpha)\cdot\sigma\left(x(\alpha,\beta)\right)+\sigma^{-1}\left(\beta\right)\cdot x(\alpha,\beta)=1,
\]
if one exists.

Recall that $\left(h_1,h_2\right)=\left(\alpha+\left(\frac{1}{2}+\frac{\sqrt{5}}{2}\right)\beta,\alpha+\left(\frac{1}{2}-\frac{\sqrt{5}}{2}\right)\beta\right)$, which is computed in Example \ref{example-2-2}.

Obviously, $m=\max\{\deg A,\sigma^{-1}\deg B\}-\deg C=1-0=1>0$.
Then do $q_0(\alpha,\beta)=h_1=\alpha+\left(\frac{1}{2}+\frac{\sqrt{5}}{2}\right)\beta$, we can find that there is an invariant $p=\left(\frac{1}{2}+\frac{\sqrt{5}}{2}\right)^2$ such that
\[
(\beta-\alpha)\cdot \sigma\left(\frac{p(\alpha,\beta)}{q_0(\alpha,\beta)}\right)+\alpha\cdot \frac{p(\alpha,\beta)}{q_0(\alpha,\beta)}=1.
\]
Now we get $x(\alpha,\beta)=\frac{\left(\frac{1}{2}+\frac{\sqrt{5}}{2}\right)^2}{\alpha+\left(\frac{1}{2}+\frac{\sqrt{5}}{2}\right)\beta}$, hence
\[
g(\alpha,\beta)=\frac{\sigma^{-1}\left(B(\alpha,\beta)\right)\cdot x(\alpha,\beta)}{C(\alpha,\beta)}\cdot f(\alpha,\beta)
=\frac{\frac{3}{2}+\frac{\sqrt{5}}{2}}{(\beta-\alpha)\cdot\left(\alpha+\left(\frac{1}{2}+\frac{\sqrt{5}}{2}\right)\beta\right)}
\]
is a solution of (\ref{eq-3-12}), which implies
\[
\sum_{n=1}^k\frac{(-1)^{n-1}}{L_{n-1}L_{n}}=(-1)^{k-1}\frac{\frac{1}{2}+\frac{\sqrt{5}}{2}}{L_k\cdot\left(L_k+\left(\frac{1}{2}+\frac{\sqrt{5}}{2}\right)L_{k+1}\right)}+\frac{\sqrt{5}}{10}.\]
Setting $k\rightarrow\infty$, we obtain
\[
\sum_{n=1}^\infty\frac{(-1)^{n-1}}{L_{n-1}L_{n}}=\frac{\sqrt{5}}{10},
\]
it is a result of Brousseau \citep[Eq.(3)]{Brousseau-1969b}.
\qed\end{example}

\begin{example}
Consider the solution $g(\alpha, \beta)$ in Example \ref{example-3-2}.

According to Step 3, looking for nonzero solutions $x(\alpha, \beta)\in\mathbb{F}(\alpha, \beta)$ of
\begin{align}\label{eq-3-14}
(\beta-\alpha)\cdot\sigma\left(x(\alpha,\beta)\right)-\sigma^{-1}\left(\alpha+2\beta\right)\cdot x(\alpha,\beta)=\alpha^2.
\end{align}

Recall that $\left(h_1,h_2\right)=\left(\alpha+(1+\sqrt{5})\beta, \alpha+(1-\sqrt{5})\beta\right)$, which is computed in Example \ref{example-2-4}.

Since $m=\max\{\deg A,\sigma^{-1}\deg B\}-\deg C=1-2=-1<0$, the problem is reduced to find a nonzero polynomial solution of (\ref{eq-3-14}), then using the method of finding a polynomial solution mentioned, we get $x(\alpha,\beta)=\frac{1}{2}\left(\alpha-\beta\right)$, hence
\[
g(\alpha,\beta)=\frac{\sigma^{-1}\left(B(\alpha,\beta)\right)\cdot x(\alpha,\beta)}{C(\alpha,\beta)}\cdot f(\alpha,\beta)
=\frac{1}{2}\frac{\alpha-\beta}{\alpha(\beta-2\alpha)}
\]
is a solution of (\ref{eq-3-13}), which implies
\[
\sum_{n=2}^k\frac{P_n}{P_{n-1}P_{n+1}}=\frac{1}{2}\left(\frac{1}{P_{k+1}}+\frac{1}{P_k}-\frac{3}{2}\right).
\]
Setting $k\rightarrow\infty$, we obtain
\[
\sum_{n=2}^\infty\frac{P_n}{P_{n-1}P_{n+1}}=-\frac{3}{4},
\]
it is a result of Koshy \citep[Example 10.3]{Koshy-2014}.
\qed\end{example}

\section{Nontrivial Case}\label{sec-4}
In this section, for the bivariate difference field $(\mathbb{F}(\alpha,\beta),\sigma)$ with  $\sigma|_\mathbb{F}={\rm id}$ and the ratio $\lambda_1/\lambda_2\in\mathbb{F}$ of eigenvalues of matrix $\left(
    \begin{array}{cc}
      0 & u \\
      1 & v \\
    \end{array}
  \right)$ is not root of unity, we consider the case $a, b\in\mathbb{F}(\alpha,\beta)$ in the difference equation (\ref{eq-dif}). We will give an algorithm for finding the solution and show some examples.

Actually, we only need consider the problem: For given polynomials $a(\alpha, \beta)$, $b(\alpha, \beta)$, $f(\alpha, \beta)\in\mathbb{F}[\alpha, \beta]$, find a rational function $g\in\mathbb{F}(\alpha, \beta)$  such that
\begin{align}\label{eq-nontrivial}
a(\alpha, \beta)\cdot\sigma\left(g(\alpha, \beta)\right)+b(\alpha, \beta)\cdot g(\alpha, \beta)=f(\alpha, \beta).
\end{align}
The steps of the algorithm are:
\begin{algorithm}
\caption{Compute a Solution of Nontrivial Case}
\label{alg-4}
\begin{algorithmic}[htb]
\STATE{\textbf{Input:} Parameters $u,v$ of the bivariate difference field $(\mathbb{F}(\alpha,\beta),\sigma)$ and polynomials $a(\alpha, \beta), b(\alpha, \beta), f(\alpha, \beta)$.}
\STATE{\textbf{Output:}  A rational function $g(\alpha, \beta)$ satisfying (\ref{eq-nontrivial}), if one exists; FALSE, otherwise.}
\STATE{\textbf{Step 1:} Find the finite part $\overline{q}(\alpha, \beta)$ in the splitting factorization of an universal denominator $q(\alpha, \beta)\in\mathbb{F}[\alpha, \beta]$ if one exists;
otherwise return ``FALSE" and stop.}
\STATE{\textbf{Step 2:} Find the infinite part $q_{\infty}(\alpha, \beta)$ in the splitting factorization of an universal denominator $q(\alpha, \beta)\in\mathbb{F}[\alpha, \beta]$ if one exists;
otherwise return ``FALSE" and stop.}
\raggedright{\STATE{\textbf{Step 3:} Find a nonzero polynomial solution $p(\alpha, \beta)\in\mathbb{F}[\alpha, \beta]$ such that $a(\alpha,\beta)\sigma\left(\frac{p(\alpha,\beta)}{q(\alpha, \beta)}\right)-b(\alpha,\beta)\frac{p(\alpha,\beta)}{q(\alpha, \beta)}=f(\alpha,\beta)$  if one exists;
otherwise return ``FALSE" and stop.}}
\STATE{\textbf{Step 4:} Return ``$\frac{p(\alpha, \ \beta)}{q(\alpha,\ \beta)}$'' and stop.}
  \end{algorithmic}
\end{algorithm}

Note that the unanswered steps are the Step 1 and Step 2, in which the problem of solving for rational solutions is reduced to the problem of solving for polynomial solutions, then for the Step 3, we can use the method of finding a polynomial solution mentioned in \citep{Guan-2021} to solve it.

In order to find a universal denominator of (\ref{eq-nontrivial}), we give a lemma first.
\begin{lemma}\label{lem-untrivial}
Let $u_0,v_0,u_1,v_1\in\mathbb{F}[\alpha, \beta]$ be polynomials satisfying $\gcd(u_0,v_0)=\gcd(u_1,$ $v_1)=1$.
If $\frac{u_0}{v_0}+\frac{u_1}{v_1}\in\mathbb{F}[\alpha,\beta]$, then $\frac{v_0}{v_1}\in\mathbb{F}$.
\end{lemma}
\begin{proof}
Let $g=\gcd(v_0,v_1)$, then
$$\frac{u_0}{v_0}+\frac{u_1}{v_1}=\frac{u_0v_1+u_1v_0}{v_0v_1}=\frac{u_0v_1/g+u_1v_0/g}{v_0v_1/g}.$$

Since $\gcd(u_0v_1/g+u_1v_0/g,v_0/g)=\gcd(u_0v_1/g,v_0/g)$ and $\gcd(u_0,v_0)=1$, $\gcd(u_0v_1/g$ $+u_1v_0/g, v_0/g)=1$. Then by $\frac{u_0}{v_0}+\frac{u_1}{v_1}\in\mathbb{F}[\alpha,\beta]$, we have $v_0/g\in\mathbb{F}$.
Similarly, $v_1/g\in\mathbb{F}$. Hence $\frac{v_0}{v_1}=\frac{v_0}{g}\frac{g}{v_1}\in \mathbb{F}.$
\end{proof}

Now, suppose $\frac{p(\alpha,\beta)}{q(\alpha,\beta)}$ is a rational solution to (\ref{eq-nontrivial}), where $p,q\in\mathbb{F}[\alpha,\beta]$ are relatively prime, and let $q=q_{\infty}\overline{q}$ be the splitting factorization of $q$ with respect to $\sigma$. We will show how to find the finite part $\overline{q}\in\mathbb{F}[\alpha,\beta]$ and the infinite part $q_{\infty}\in\mathbb{F}[\alpha,\beta]$ in following subsections, respectively.

\subsection{Step 1 of Algorithm \ref{alg-4}}
In this subsection, we consider the finite part $\overline{q}\in\mathbb{F}[\alpha,\beta]$ in the splitting factorization of $q$ with respect to $\sigma$ such that $\frac{p(\alpha,\beta)}{q(\alpha,\beta)}$ is a rational solution to (\ref{eq-nontrivial}), where $p,q\in\mathbb{F}[\alpha,\beta]$ are relatively prime.

Let $u:=\gcd(a,\sigma q)$ and $v:=\gcd(b,q)$, then by (\ref{eq-nontrivial}), we have
\[
\frac{(a/u)\sigma p}{\sigma q/u}+\frac{(b/v)p}{q/v}\in\mathbb{F}[\alpha,\beta].
\]
Since $p,q$ are relatively prime, $\gcd((a/u)\sigma p, \sigma q/u)=\gcd((b/v)p, q/v)$=1.
Lemma \ref{lem-untrivial} implies that $\sigma q/u$ and $q/v$ are the same up to a scalar multiple, so suppose
\begin{align}\label{eq-4-1}
\frac{\sigma q/u}{q/v}=k
\end{align}
where $k\in\mathbb{F}$. Recall that $q=q_{\infty}\overline{q}$ is the splitting factorization of $q$ with respect to $\sigma$, then (\ref{eq-4-1}) could be written as
\[
\sigma q_{\infty} \cdot \sigma \overline{q}\cdot v=k\cdot q_{\infty}\cdot \overline{q}\cdot u,
\]
it shows
\begin{align}\label{eq-4-2}
  q_{\infty} \cdot \overline{q}\mid\sigma^{-1}q_{\infty} \cdot \sigma^{-1}\overline{q}\cdot\sigma^{-1}u, \nonumber\\
 q_{\infty}\cdot\overline{q}\mid\sigma q_{\infty} \cdot \sigma \overline{q}\cdot v .
\end{align}

Obviously, there exists $M\in$ $\Spr_{\sigma}(q_{\infty})\backslash$ $\Spr_{\sigma}(\overline{q})$. Using the relations (\ref{eq-4-2}) repeatedly, one finds that
\begin{align*}
 q_{\infty}\cdot\overline{q}\mid\sigma^{-M} q_{\infty}\cdot\sigma^{-M}\overline{q}\cdot \sigma^{-1}u \cdot \sigma^{-2}u \cdots \sigma^{-M}u, \nonumber\\
q_{\infty}\cdot \overline{q}\mid\sigma^M q_{\infty} \cdot \sigma^M \overline{q}\cdot \sigma^{M-1}v \cdot \sigma^{M-2}v \cdots v ,
 \end{align*}
moreover by the choice of $M$, we have $\overline{q}\nmid\sigma^M\overline{q}$, $\overline{q}\nmid\sigma^{-M}\overline{q}$ and $\sigma^Mq_{\infty}=\lambda q_{\infty}$ where $\lambda\in\mathbb{F}$, so
\begin{align*}
\overline{q}\mid\sigma^{-1}u \cdot \sigma^{-2}u \cdots \sigma^{-M}u,\nonumber\\
\overline{q}\mid\sigma^{M-1}v \cdot \sigma^{M-2}v \cdots v.
\end{align*}
Let $u=u_{\infty}\overline{u}$ and $v=v_{\infty}\overline{v}$ be the splitting factorization of $u,v$ with respect to $\sigma$, respectively. Then by the definition of splitting factorization, we have
 \begin{align*}
\overline{q}\mid\sigma^{-1}\overline{u} \cdot \sigma^{-2}\overline{u} \cdots \sigma^{-M}\overline{u}, \nonumber\\
\overline{q}\mid\sigma^{M-1}\overline{v} \cdot\sigma^{M-2}\overline{v} \cdots \overline{v}.
\end{align*}
Let $a=a_{\infty}\overline{a}$ and $b=b_{\infty}\overline{b}$ are the splitting factorization of $a,b$ with respect to $\sigma$, respectively. Then by the definition of $u,v$, we have
\begin{align*}
\overline{q}\mid\sigma^{-1}\overline{a} \cdot \sigma^{-2}\overline{a} \cdots \sigma^{-M}\overline{a}, \nonumber\\
\overline{q}\mid\sigma^{M-1}\overline{b} \cdot \sigma^{M-2}\overline{b} \cdots \overline{b}.
\end{align*}
Hence
\[
\overline{q}\mid\gcd(\sigma^{-1}\overline{a} \cdot \sigma^{-2}\overline{a} \cdots \sigma^{-M}\overline{a},\  \sigma^{M-1}\overline{b} \cdot \sigma^{M-2}\overline{b} \cdots \overline{b}).
\]
Moreover, using the definition of $\gcd$, we have
\[
\gcd(\sigma^{-1}\overline{a} \cdot \sigma^{-2}\overline{a} \cdots \sigma^{-M}\overline{a},\ \sigma^{M-1}\overline{b} \cdot \sigma^{M-2}\overline{b} \cdots \overline{b})\mid
\prod_{i=1}^M\prod_{j=0}^{M-1}\gcd(\sigma^{-i}\overline{a},\ \sigma^{j}\overline{b}),
\]
then
\[
\overline{q}\mid \prod_{i=1}^M\prod_{j=0}^{M-1}\sigma^{-i}\left(\gcd(\overline{a},\ \sigma^{i+j}\overline{b})\right).
\]
By Theorem \ref{thm-prime}, we have $\Spr_{\sigma}(\overline{a},\overline{b})$ is finite.
Suppose $\Spr_{\sigma}(\overline{a},\overline{b})=\{m_1,m_2,\ldots,m_N\}$ is the spread set of $\overline{a}$ and $\overline{b}$ with respect to $\sigma$, then by the definition of spread, we have
\[
 \prod_{i=1}^M\prod_{j=0}^{M-1}\sigma^{-i}\left(\gcd(\overline{a},\ \sigma^{i+j}\overline{b})\right)\mid
 \prod_{k=1}^N\prod_{i=1}^{m_k}\sigma^{-i}\left(\gcd(\overline{a},\ \sigma^{m_k}\overline{b})\right),
\]
and
\begin{align}\label{eq-4-3}
\overline{q}\mid \prod_{k=1}^N\prod_{i=1}^{m_k}\sigma^{-i}\left(\gcd(\overline{a},\ \sigma^{m_k}\overline{b})\right).
\end{align}

Next, we will show a polynomial which could be divisible by $\overline{q}$, and its total degree is less than or equal to $\deg\left(\prod_{k=1}^N\prod_{i=1}^{m_k}\sigma^{-i}\left(\gcd(\overline{a},\ \sigma^{m_k}\overline{b})\right)\right)$. For this, we do the following computations.
\begin{breakablealgorithm}
\caption{Step 1 of Algorithm \ref{alg-4}}
\label{alg-5}
\raggedright{\textbf{Step 1.1:} Let $a=a_{\infty}\overline{a}$ and $b=b_{\infty}\overline{b}$ be the splitting factorization of $a,b$ with respect to $\sigma$, respectively.}\\
\raggedright{\textbf{Step 1.2:} Compute $\Spr_{\sigma}(\overline{a},\overline{b})=\{m_1,m_2,\ldots,m_N\}$ and $m_1<m_2<\ldots<m_N$.}\\
\raggedright{\textbf{Step 1.3:} $a_0:=\overline{a}, b_0:=\overline{b}$;}
\begin{algorithmic}[htb]
 \FOR{$1\leq i \leq N$}
  \STATE{$s_i:=\gcd(a_{i-1},\sigma^{m_i}b_{i-1});\ a_{i}:=\frac{a_{i-1}}{s_i};\ b_{i}=\frac{b_{i-1}}{\sigma^{-m_i}s_i}.$}
\ENDFOR
\STATE{$\overline{q}(\alpha,\beta):=\prod_{i=1}^N\prod_{j=1}^{m_i}\sigma^{-j}s_i.$}
\end{algorithmic}
\end{breakablealgorithm}

Using these notations, we have the following lemma:

\begin{lemma}\label{lem-4-1}
Let $a=a_{\infty}\overline{a}$ and $b=b_{\infty}\overline{b}$ be the splitting factorization of $a,b$ with respect to $\sigma$, respectively, $\Spr_{\sigma}(\overline{a},\overline{b})=\{m_1,m_2,\ldots,m_N\}$ and $m_1<m_2<\ldots<m_N$.
For each $1\leq k\leq N$,
\[
\gcd(a_{k-1},\sigma^{m_k}b_{k-1})=\gcd(\overline{a}, \sigma^{m_k}\overline{b}),
\]
where $a_i, b_i$ are defined in the Step 1.3.

\begin{proof}
It is easy to check that for each $1\leq k\leq N$,
$\gcd(a_{k-1},\sigma^{m_{k}}b_{k-1})\mid\gcd(\overline{a}, \sigma^{m_{k}}\overline{b})$.

On the other hand, by Algorithm \ref{alg-5}, we have $\overline{a}=a_{k-1}\prod_{i=1}^{k-1}s_i$ and $\overline{b}=b_{k-1}\prod_{i=1}^{k-1}\sigma^{-m_i}s_i$ for each $1\leq k\leq N$, then
\begin{align*}
\gcd(\overline{a},\sigma^{m_{k}}\overline{b})\mid &\gcd\left(a_{k-1}, \sigma^{m_{k}}b_{k-1}\right) \cdot
\gcd\left(a_{k-1}, \prod_{i=1}^{k-1}\sigma^{m_{k}-m_i}s_i\right) \\
&\cdot \gcd\left(\prod_{i=1}^{k-1}s_i, \sigma^{m_{k}}b_{k-1}\right)
 \cdot \gcd\left(\prod_{i=1}^{k-1}s_i, \prod_{i=1}^{k-1}\sigma^{m_{k}-m_i}s_i\right)  .
\end{align*}

Now, we only need to show that $\deg\left( \gcd\left(a_{k-1}, \prod_{i=1}^{k-1}\sigma^{m_{k}-m_i}s_i\right)\right)$ =$\deg\left(\gcd\left(\prod_{i=1}^{k-1}s_i, \sigma^{m_{k}}b_{k-1}\right)\right)$ =$\deg\left(\gcd\left(\prod_{i=1}^{k-1}s_i, \prod_{i=1}^{k-1}\sigma^{m_{k}-m_i}s_i\right)\right)=0$.

Suppose $\gcd\left(a_{k-1}, \prod_{i=1}^{k-1}\sigma^{m_{k}-m_i}s_i\right)=t$, where $\deg t\neq0$, then there exist polynomials $d_1,d_2\in\mathbb{F}[\alpha,\beta]$ such that
$a_{k-1}=t\cdot d_1$ and
$\prod_{i=1}^{k-1}\sigma^{m_{k}-m_i}s_i=t\cdot d_2$.
Moreover, we have
\[
\overline{a}=a_{k-1}\prod_{i=1}^{k-1}s_i=t\cdot d_1\cdot\prod_{i=1}^{k-1}s_i,
\]
and
\[
\sigma^{m_N+m_{k}}\overline{b}=\sigma^{m_N+m_{k}}b_{k-1}\cdot\prod_{i=1}^{k-1}\sigma^{m_N+m_{k}-m_i}s_i=\sigma^{m_N+m_{k}}b_k\cdot t\cdot d_2\cdot\prod_{i=1}^{k-1}\sigma^{m_N}s_i.
\]
Hence $t\mid\gcd(\overline{a},\sigma^{m_N+m_{k}}\overline{b})$,
which implies $m_N+m_{k}\in\Spr_{\sigma}(\overline{a},\overline{b})$, a contradiction.

Next, suppose $\gcd\left(\prod_{i=1}^{k-1}s_i, \sigma^{m_{k}}b_{k-1}\right)=t$, where $\deg t\neq0$, then there exist polynomials $d_1,d_2\in\mathbb{F}[\alpha,\beta]$ such that
$\prod_{i=1}^{k-1}s_i=t\cdot d_1$ and $\sigma^{m_{k}}b_{k-1}=t\cdot d_2$.
Moreover, we have
\[\overline{a}=a_{k-1}\prod_{i=1}^{k-1}s_i=a_{k-1}\cdot t\cdot d_1\]
and
\[\sigma^{m_N+m_{k}}\overline{b}=\sigma^{m_N}\overline{b}\cdot\sigma^{m_{k}}b_{k-1}\cdot\prod_{i=1}^{k-1}\sigma^{m_{k}-m_i}s_i=\sigma^{m_N}\overline{b}\cdot t\cdot d_2\cdot\prod_{i=1}^{k-1}\sigma^{m_{k}-m_i}s_i.
\]
Hence $t\mid\gcd(\overline{a},\sigma^{m_N+m_{k}}\overline{b})$, which implies $m_N+m_{k}\in\Spr_{\sigma}(\overline{a},\overline{b})$, a contradiction.

Finally, suppose $\gcd\left(\prod_{i=1}^{k-1}s_i, \prod_{i=1}^{k-1}\sigma^{m_{k}-m_i}s_i\right)=t$, where $\deg t\neq0$, then there exist polynomials $d_1,d_2\in\mathbb{F}[\alpha,\beta]$ such that
$\prod_{i=1}^{k-1}s_i=t\cdot d_1$ and $\prod_{i=1}^{k-1}\sigma^{m_{k}-m_i}s_i=t\cdot d_2$.
Moreover, we have
\[
\overline{a}=a_{k-1}\prod_{i=1}^{k-1}s_i=a_{k-1}\cdot t\cdot d_1
\]
and
\[
\sigma^{m_N+m_{k}}\overline{b}=\sigma^{m_N+m_{k}}b_k\cdot\prod_{i=1}^{k-1}\sigma^{m_N+m_{k}-m_i}s_i=\sigma^{m_N+m_{k}}b_k\cdot t\cdot d_2\cdot\prod_{i=1}^{k-1}\sigma^{m_N}s_i.
\]
Hence $t\mid\gcd(\overline{a},\sigma^{m_N+m_{k}}\overline{b})$, which implies $m_N+m_{k}\in\Spr_{\sigma}(\overline{a},\overline{b})$, a contradiction.

As a consequence, $\gcd\left(\overline{a},\sigma^{m_{k}}\overline{b}\right)\mid\gcd\left(a_{k-1}, \sigma^{m_{k}}b_{k-1}\right)$,
and this concludes the proof of the lemma.
\end{proof}
\end{lemma}

Combining Lemma \ref{lem-4-1} and (\ref{eq-4-3}), we obtain the following theorem immediately.
\begin{theorem}\label{thm-4-1}
Let $g(\alpha,\beta)=\frac{p(\alpha,\beta)}{q(\alpha,\beta)}$ be a solution to (\ref{eq-nontrivial}), where $p,q\in\mathbb{F}[\alpha,\beta]$ are relatively prime. Then for the finite part in the splitting factorization $q=q_{\infty}\overline{q}$ with respect to $\sigma$, we have
\begin{align*}
\overline{q}\mid\prod_{i=1}^N\prod_{j=1}^{m_i}\sigma^{-j}s_i,
\end{align*}
where $s_i$ are defined in the Step 1.3.
\end{theorem}

\begin{example}\label{example-4-3}
Consider the difference equation (\ref{eq-nontrivial}) involving the Pell numbers $P_n$ in the difference field $\mathbb{R}[\alpha,\beta]$, which is defined in Example \ref{example-2-3}.
Given $a=\beta$, $b=\alpha$ and $f=\alpha+3\beta$, find a solution $g\in\mathbb{F}(\alpha,\beta)$ such that
\begin{align}\label{eq-4-7}
\beta\cdot\sigma g+\alpha\cdot g=
\alpha+3\beta.
\end{align}

According to Step 1, we can compute that $\overline{a}=a,\overline{b}=b$ and $\Spr_{\sigma}(\overline{a},\overline{b})=\{1\}$, then
\[s_1=\gcd\left(a,b\right)=\beta,\  a_1=a/s_1=1,\ b_1=b/\sigma^{-1}s_1=1.\]

Then by Theorem \ref{thm-4-1}, we have
\[\overline{q}\mid\prod_{i=1}^1\prod_{j=1}^{m_i}\sigma^{-j}s_i=\alpha,
\]
hence the finite part of the universal denominator is $\alpha$.
\qed\end{example}

\begin{example}\label{example-4-1}
Consider the difference equation (\ref{eq-nontrivial}) involving the Fibonacci numbers $F_n$ in the difference field $\mathbb{R}[\alpha,\beta]$, which is defined in Example \ref{example-2-1}.
Given $a=\alpha+\beta$, $b=\alpha\beta$ and $f=\alpha^3+\beta^2-\alpha\beta-\alpha-\beta$, find a solution $g\in\mathbb{F}(\alpha,\beta)$ such that
\begin{align}\label{eq-4-5}
(\alpha+\beta)\cdot\sigma g+(\alpha\beta)\cdot g=
\alpha^3+\beta^2-\alpha\beta-\alpha-\beta.
\end{align}

According to Step 1, we can compute that $\overline{a}=a,\overline{b}=b$ and $\Spr_{\sigma}(\overline{a},\overline{b})=\{1,2\}$, then
\[s_1=\gcd\left(a,\sigma b\right)=\alpha+\beta,\  a_1=a/s_1=1,\  b_1=b/\sigma^{-1} s_1=\alpha,\]
and
\[s_2=\gcd(a_1,\sigma^2 b_1)=1,\  a_2=a_1/s_2=1,\ b_2=b_1/\sigma^{-2}s_2=\alpha.\]

Then by Theorem \ref{thm-4-1}, we have
\[\overline{q}\mid\prod_{i=1}^2\prod_{j=1}^{m_i}\sigma^{-j}s_i=\beta,
\]
hence the finite part of the universal denominator is $\beta$.
\qed\end{example}

\begin{example}\label{example-4-2}
Consider the difference equation (\ref{eq-nontrivial}) involving the Lucas numbers $L_n$ in the difference field $\mathbb{R}[\alpha,\beta]$, which is defined in Example \ref{example-2-1}.
Given $a=\alpha^2(\alpha-\beta)(\alpha+2\beta)$, $b=(-\alpha^3)(\alpha+\beta)$ and $f=\alpha^2$, find a solution $g\in\mathbb{F}(\alpha,\beta)$ such that
\begin{align}\label{eq-4-6}
\alpha^2(\alpha-\beta)(\alpha+2\beta)\cdot\sigma g+(-\alpha^3)(\alpha+\beta)\cdot g=
\alpha^2.
\end{align}

According to Step 1, we can compute that $\overline{a}=a,\overline{b}=b$ and $\Spr_{\sigma}(\overline{a},\overline{b})=\{0,1,3\}$, then
\[
s_1=\gcd\left(a,b\right)=\alpha^2,\  a_1=a/s_1=(\alpha-\beta)(\alpha+2\beta),\  b_1=b/s_1=-\alpha(\alpha+\beta),
\]
\[
s_2=\gcd\left(a_1,\sigma b_1\right)=\alpha+2\beta,\  a_2=a_1/s_2=\alpha-\beta,\  b_2=b_1/\sigma^{-1} s_2=-\alpha,
\]
and
\[s_3=\gcd(a_2,\sigma^3 b_2)=1,\  a_3=a_2/s_2=\alpha-\beta,\ b_3=b_2/\sigma^{-3}s_3=-\alpha.\]

Then by Theorem \ref{thm-4-1}, we have
\[\overline{q}\mid\prod_{i=1}^3\prod_{j=1}^{m_i}\sigma^{-j}s_i=\alpha+\beta,
\]
hence the finite part of the universal denominator is $\alpha+\beta$.
\qed\end{example}
\subsection{Remaining Steps of Algorithm \ref{alg-4}}
Now, we consider the infinite part $q_{\infty}\in\mathbb{F}[\alpha,\beta]$, recall that $\frac{p(\alpha,\beta)}{q(\alpha,\beta)}$ is a rational solution to (\ref{eq-nontrivial}) and $q=q_{\infty}\overline{q}$ is the splitting factorization of $q$ with respect to $\sigma$, where $p,q\in\mathbb{F}[\alpha,\beta]$ are relatively prime.

By Theorem \ref{thm-2-2}, if $q_{\infty}$ is homogenous polynomial, then there exist $m\in\mathbb{N}$ and $0\leq i\leq m$ such that
\[
q_{\infty}(\alpha, \beta)|h_1(\alpha,\beta)^{i}h_2(\alpha,\beta)^{m-i},
\]
where $(h_1(\alpha,\beta),h_2(\alpha,\beta))=(\alpha, \beta)(X_1,X_2)$ and $X_1, X_2$ are two eigenvectors of matrix
$\left(
    \begin{array}{cc}
      0 & u \\
      1 & v \\
    \end{array}
  \right)$.
However we could not get the upper bound of $m$, hence we just assume that $m\leq\max\{\deg a,$ $\deg(\sigma^{-1}b)\}$, then using Step 3 to compute polynomial $p(\alpha, \beta)$.

Recall that $\overline{q}=\prod_{i=1}^N\prod_{j=1}^{m_i}\sigma^{-j}s_i$ in the Step 1. According to this, we give the following computation of the remaining steps.
\begin{breakablealgorithm}
\caption{Remaining Steps of Algorithm \ref{alg-4}}
\label{alg-6}
\begin{algorithmic}[htp]
\STATE {1. $A:=
  \left(
    \begin{array}{cc}
      0 & u \\
      1 & v \\
    \end{array}
  \right)$. Then compute the eigenvectors $X_1, X_2$ of matrix $A$ and\\ \quad $(h_1,h_2):=(\alpha,\beta)(X_1,X_2)$.}
\STATE {2. $m:=\max\{\deg a,\deg b\}-\deg f-\deg \overline{q}$.}
\IF{$m\leq0$}
  \STATE{ Find a nonzero polynomial solution $p(\alpha,\beta)$ of
\[
a(\alpha,\beta)\cdot\sigma\left(\frac{p(\alpha,\beta)}{\overline{q}(\alpha,\beta)}\right)+ b(\alpha,\beta)\cdot \frac{p(\alpha,\beta)}{\overline{q}(\alpha,\beta)}=f(\alpha,\beta).
\]
If one exists return ``$\frac{p(\alpha,\beta)}{\overline{q}(\alpha,\beta)}$''; otherwise return ``FALSE" and stop.}
\ELSE[$m>0$]
  \FOR{each $0\leq k\leq m+1$}\vspace{-5mm}
     \STATE {\FOR{$0\leq i\leq k$}
                  \STATE {$q_i(\alpha,\beta):=h_1^{k-i}h_2^{i};$}\\
                   \STATE {find a nonzero polynomial solution $p(\alpha,\beta)$ such that
\[
a(\alpha,\beta)\cdot\sigma\left(\frac{p(\alpha,\beta)}{\overline{q}(\alpha,\beta)q_i(\alpha,\beta)}\right)+b(\alpha,\beta)\cdot\frac{p(\alpha,\beta)}{\overline{q}(\alpha,\beta)q_i(\alpha,\beta)}=f(\alpha,\beta).\]}\ENDFOR}
\ENDFOR
\STATE{ If one exists return ``$\frac{p(\alpha,\beta)}{\overline{q}(\alpha,\beta)q_i(\alpha,\beta)}$'';
otherwise return ``FALSE" and stop.}
\ENDIF
\end{algorithmic}
\end{breakablealgorithm}

Note that return ``FALSE'' does not mean there is no nonzero solution $g(\alpha, \beta)$ of (\ref{eq-nontrivial}).

\begin{example}
In the difference field $\mathbb{R}[\alpha,\beta]$ involving the Pell numbers $P_n$ which is defined in Example \ref{example-2-3}, consider the infinite part of the universal denominator from Example \ref{example-4-3}.

Recall that $\overline{q}(\alpha,\beta)=\alpha$ by the Step 1,
\[
  A =
  \left(
    \begin{array}{cc}
      0 & 1 \\
      1 & 2 \\
    \end{array}
  \right),
  \]
and $(h_1,h_2):=\left(\alpha+\left(1+\sqrt{2}\right)\beta, \alpha+\left(1-\sqrt{2}\right)\beta\right)$ by Example \ref{example-2-4}.

Since $m=\max\{\deg a,\deg b\}-\deg f-\deg \overline{q}=1-1-0=-1<0$, the problem is reduced to find a nonzero polynomial solution $p(\alpha,\beta)$ of
\vspace{2mm}
\[
\beta\cdot\sigma\left(\frac{p(\alpha,\beta)}{\overline{q}(\alpha,\beta)}\right)+\alpha\cdot \frac{p(\alpha,\beta)}{\overline{q}(\alpha,\beta)}=\alpha+3\beta.
\]

Then using the method of finding a polynomial solution, we get $p(\alpha,\beta)=\beta$, hence $g(\alpha,\beta)=\frac{\beta}{\alpha}$ is a solution of (\ref{eq-4-7}).
\qed\end{example}

\begin{example}
In the difference field $\mathbb{R}[\alpha,\beta]$ involving the Fibonacci numbers $F_n$ which is defined in Example \ref{example-2-1}, consider the infinite part of the universal denominator from Example \ref{example-4-1}.

Recall that $\overline{q}(\alpha,\beta)=\beta$ by the Step 1,
\[
  A =
  \left(
    \begin{array}{cc}
      0 & 1 \\
      1 & 1 \\
    \end{array}
  \right),
  \]
and $(h_1,h_2):=\left(\alpha+\left(\frac{1}{2}+\frac{\sqrt{5}}{2}\right)\beta, \alpha+\left(\frac{1}{2}-\frac{\sqrt{5}}{2}\right)\beta\right)$ by Example \ref{example-2-2}.

Since $m=\max\{\deg a,\deg b\}-\deg f-\deg \overline{q}=2-3-1=-2<0$, the problem is reduced to find a nonzero polynomial solution $p(\alpha,\beta)$ of
\vspace{2mm}
\[
(\alpha+\beta)\cdot\sigma\left(\frac{p(\alpha,\beta)}{\overline{q}(\alpha,\beta)}\right)+(\alpha\beta)\cdot \frac{p(\alpha,\beta)}{\overline{q}(\alpha,\beta)}=\alpha^3+\beta^2-\alpha\beta-\alpha-\beta.
\]

Then using the method of finding a polynomial solution, we get $p(\alpha,\beta)=\alpha^2-\beta$, hence $g(\alpha,\beta)=\frac{\alpha^2-\beta}{\beta}$ is a solution of (\ref{eq-4-5}).
\qed\end{example}

\begin{example}
In the difference field $\mathbb{R}[\alpha,\beta]$ involving the Lucas numbers $L_n$ which is defined in Example \ref{example-2-1}, consider the infinite part of the universal denominator from Example \ref{example-4-2}.

Recall that $\overline{q}(\alpha,\beta)=\alpha+\beta$ by the Step 1,
\[
  A =
  \left(
    \begin{array}{cc}
      0 & 1 \\
      1 & 1 \\
    \end{array}
  \right),
  \]
and $(h_1,h_2):=\left(\alpha+\left(\frac{1}{2}+\frac{\sqrt{5}}{2}\right)\beta, \alpha+\left(\frac{1}{2}-\frac{\sqrt{5}}{2}\right)\beta\right)$ by Example \ref{example-2-2}.

Since $m=\max\{\deg a,\deg b\}-\deg f-\deg \overline{q}=4-2-1=1>0$, we do the following loops.

First, let $k=m=1$, for $q_1(\alpha,\beta)=h_1=\alpha+(\frac{1}{2}+\frac{\sqrt{5}}{2})\beta$, we also find that there  does not exist an invariant $p\in\mathbb{F}\setminus\{0\}$ such that
\[
\alpha^2(\alpha-\beta)(\alpha+2\beta)\cdot\sigma\left(\frac{p(\alpha,\beta)}{\overline{q}(\alpha,\beta)q_1(\alpha,\beta)}\right)
+(-\alpha^3)(\alpha+\beta)\cdot\frac{p(\alpha,\beta)}{\overline{q}(\alpha,\beta)q_1(\alpha,\beta)}=\alpha^2;
\]
Similarly, for $q_0(\alpha,\beta)=h_2=\alpha+\left(\frac{1}{2}-\frac{\sqrt{5}}{2}\right)\beta$, we can find that there does not exist an invariant $p\in\mathbb{F}\setminus\{0\}$ such that
\[
\alpha^2(\alpha-\beta)(\alpha+2\beta)\cdot\sigma\left(\frac{p(\alpha,\beta)}{\overline{q}(\alpha,\beta)q_0(\alpha,\beta)}\right)
+(-\alpha^3)(\alpha+\beta)\cdot\frac{p(\alpha,\beta)}{\overline{q}(\alpha,\beta)q_0(\alpha,\beta)}=\alpha^2;
\]
Then, let $k=m+1=2$, for $q_0(\alpha,\beta)=h_1^{2}=\alpha^2+(1+\sqrt{5})\alpha\beta+\left(\frac{3}{2}+\frac{\sqrt{5}}{2}\right)\beta^2$, we can find that there  does not exist a nonzero polynomial $p\in\mathbb{F}[\alpha,\beta]$ such that
\vspace{1mm}
\[
\alpha^2(\alpha-\beta)(\alpha+2\beta)\cdot\sigma\left(\frac{p(\alpha,\beta)}{\overline{q}(\alpha,\beta)q_0(\alpha,\beta)}\right)
+(-\alpha^3)(\alpha+\beta)\cdot\frac{p(\alpha,\beta)}{\overline{q}(\alpha,\beta)q_0(\alpha,\beta)}=\alpha^2;
\]
Similarly, for $q_1(\alpha,\beta)=h_1h_2=\alpha^2+\alpha\beta-\beta^2$, we can find that there exists a polynomial $p(\alpha,\beta)=-\alpha$ such that
\[
\alpha^2(\alpha-\beta)(\alpha+2\beta)\cdot\sigma\left(\frac{p(\alpha,\beta)}{\overline{q}(\alpha,\beta)q_1(\alpha,\beta)}\right)
+(-\alpha^3)(\alpha+\beta)\cdot\frac{p(\alpha,\beta)}{\overline{q}(\alpha,\beta)q_1(\alpha,\beta)}=\alpha^2;
\]

Now we get $g(\alpha,\beta)=-\frac{\alpha}{\left(\alpha+\beta\right)\left(\alpha^2+\alpha\beta-\beta^2\right)}$ is a solution of (\ref{eq-4-6}).
\qed\end{example}

\section{Conclusion}\label{sec-5}
In this paper, based on the summation involving a sequence satisfying a recurrence of order two,
we provide an algebraic framework for solving
a rational function $g\in\mathbb{F}(\alpha, \beta)$ satisfying the difference equation
\[a\sigma(g)+bg=f\]
in the bivariate difference field $(\mathbb{F}(\alpha, \beta), \sigma)$, where $a, b,f\in\mathbb{F}(\alpha,\beta)\setminus\{0\}$ are given.
Then, we present algorithms for finding a rational solution $g\in\mathbb{F}(\alpha, \beta)$ of the difference equation provided that (1) the infinite part of the universal denominator is  a homogeneous polynomial; (2) an upper bound for degree of the universal denominator is given;
and (3) an algorithm for determining a polynomial solution of the difference equation is also given.
Moreover, several examples are given to illustrate the applicability of the
algorithms.

\smallskip{}
\noindent \textbf{Acknowledgement.}
The authors would like to thank Peter Paule and Carsten Schneider for their advice and comments, and two referees for their careful reading and constructive suggestions. This work was supported by the China Scholarship Council (202006250120) and National Natural Science
Foundation of China (11771330, 11921001).




\end{document}